\documentclass[12pt,a4paper]{amsart}
\usepackage{amssymb, amstext, amscd, amsmath, color}
\usepackage{graphicx}

\usepackage{url}

\usepackage{hhline}

\begin{document}

\newtheorem{thm}{Theorem}[section]
\newtheorem{cor}[thm]{Corollary}
\newtheorem{prop}[thm]{Proposition}
\newtheorem{lem}[thm]{Lemma}
%
\theoremstyle{definition}
\newtheorem{rem}[thm]{Remark}
\newtheorem{defn}[thm]{Definition}
\newtheorem{note}[thm]{Note}
\newtheorem{eg}[thm]{Example}
\newcommand{\Prf}{\noindent\textbf{Proof.\ }}
\newcommand{\bx}{\hfill$\blacksquare$\medbreak}
\newcommand{\upbx}{\vspace{-2.5\baselineskip}\newline\hbox{}%
\hfill$\blacksquare$\newline\medbreak}
\newcommand{\eqbx}[1]{\medbreak\hfill\(\displaystyle #1\)\bx}

\newcommand{\FFock}{\mathcal{F}}
\newcommand{\kil}{\mathsf{k}}
\newcommand{\Hil}{\mathsf{H}}
\newcommand{\hil}{\mathsf{h}}
\newcommand{\Kil}{\mathsf{K}}
\newcommand{\Real}{\mathbb{R}}
\newcommand{\Rplus}{\Real_+}

%

\newcommand{\bC}{{\mathbb{C}}}
\newcommand{\bD}{{\mathbb{D}}}
\newcommand{\bN}{{\mathbb{N}}}
\newcommand{\bQ}{{\mathbb{Q}}}
\newcommand{\bR}{{\mathbb{R}}}
\newcommand{\bT}{{\mathbb{T}}}
\newcommand{\bX}{{\mathbb{X}}}
\newcommand{\bZ}{{\mathbb{Z}}}
\newcommand{\bH}{{\mathbb{H}}}
\newcommand{\BH}{{\B(\H)}}
\newcommand{\bsl}{\setminus}
\newcommand{\ca}{\mathrm{C}^*}
\newcommand{\cstar}{\mathrm{C}^*}
\newcommand{\cenv}{\mathrm{C}^*_{\text{env}}}
\newcommand{\rip}{\rangle}
\newcommand{\ol}{\overline}
\newcommand{\td}{\widetilde}
\newcommand{\wh}{\widehat}
\newcommand{\sot}{\textsc{sot}}
\newcommand{\wot}{\textsc{wot}}
\newcommand{\wotclos}[1]{\ol{#1}^{\textsc{wot}}}
 \newcommand{\A}{{\mathcal{A}}}
 \newcommand{\B}{{\mathcal{B}}}
 \newcommand{\C}{{\mathcal{C}}}
 \newcommand{\D}{{\mathcal{D}}}
 \newcommand{\E}{{\mathcal{E}}}
 \newcommand{\F}{{\mathcal{F}}}
 \newcommand{\G}{{\mathcal{G}}}
\renewcommand{\H}{{\mathcal{H}}}
 \newcommand{\I}{{\mathcal{I}}}
 \newcommand{\J}{{\mathcal{J}}}
 \newcommand{\K}{{\mathcal{K}}}
\renewcommand{\L}{{\mathcal{L}}}
 \newcommand{\M}{{\mathcal{M}}}
 \newcommand{\N}{{\mathcal{N}}}
\renewcommand{\O}{{\mathcal{O}}}
\renewcommand{\P}{{\mathcal{P}}}
 \newcommand{\Q}{{\mathcal{Q}}}
 \newcommand{\R}{{\mathcal{R}}}
\renewcommand{\S}{{\mathcal{S}}}
 \newcommand{\T}{{\mathcal{T}}}
 \newcommand{\U}{{\mathcal{U}}}
 \newcommand{\V}{{\mathcal{V}}}
 \newcommand{\W}{{\mathcal{W}}}
 \newcommand{\X}{{\mathcal{X}}}
 \newcommand{\Y}{{\mathcal{Y}}}
 \newcommand{\Z}{{\mathcal{Z}}}

\newcommand{\sgn}{\operatorname{sgn}}
\newcommand{\rank}{\operatorname{rank}}

\newcommand{\Isom}{\operatorname{Isom}}

\newcommand{\qIsom}{\operatorname{q-Isom}}

\newcommand{\sIsom}{\operatorname{s-Isom}}

\newcommand{\sep}{\operatorname{sep}}




 \title[Nets and Meshes]{String-node nets and  meshes}

\author[S.C. Power and B. Schulze]{S.C. Power and B. Schulze}
\address{Dept.\ Math.\ Stats.\\ Lancaster University\\
Lancaster LA1 4YF \\U.K. }
\email{b.schulze@lancaster.ac.uk}
\email{s.power@lancaster.ac.uk}


\thanks{SCP supported by EPSRC grant  EP/J008648/1.}
\thanks{BS supported by EPSRC grant  EP/M013642/1.}

\begin{abstract}
New classes of infinite bond-node  structures are introduced, namely \emph{string-node nets} and \emph{meshes},  
a mesh being a  string-node net for which the nodes are dense in the strings. Various construction schemes are given including the minimal extension of a (countable) line segment net  by a countable scaling group. A \emph{linear mesh}  has strings that are straight lines and nodes given by the intersection points of these lines. Classes of meshes, such as the \emph{regular  meshes} in $\bR^2$ and $\bR^3$, are defined and classified. 
String-length preserving motions  are also determined for a number of fundamental examples and contrasting flexing and rigidity properties are obtained with respect to noncrossing motions in the space of \emph{smooth meshes}.
\end{abstract}

\date{}
\maketitle

\section{introduction}\label{s:intro}
A  {bar-joint framework} is an abstract structure in a Euclidean space composed of rigid bars which are connected at joints where there is free articulation (Asimow and Roth \cite{asi-rot},\cite{asi-rot-2}, Whiteley \cite{whi-handbook}). A  {tensegrity framework}
is a similar structure in which there are also cables, or strings, (and sometimes struts) which constrain the motion of the joints (Gruenbaum and  Shephard  \cite{gru-lostlectures}, Roth and Whiteley \cite{rot-whi}, Snelson \cite{sne}).
In what follows we move in a more mathematical direction and introduce new classes of distance-constrained  structures, namely  {string-node nets} and  {meshes}. By a mesh we mean an infinite string-node net for which the nodes are dense in the strings, which in turn may possibly be dense in the ambient space. 
A planar mesh could be likened to a fabric or material, made of infinitesimally thin threads, except that the threads are joined pointwise rather than being interlaced, and the join points are dense on each thread. The primary constraint for flexes and motions is that the internodal string lengths must be preserved.

In fact in Sections 2, 3 and 4 we shall be concerned solely with string-node nets and meshes as \emph{static} geometrical constructs. 
These structures are particularly interesting in the periodic case since their discrete counterparts provide the fundamental bond-node nets in crystallography and in reticular chemistry. In this connection Delgado-Friedrichs,  O'Keeffe and  Yaghi \cite{del-et-al-1} have observed that there are exactly $5$ regular discrete periodic nets in three dimensions. The subsequent papers \cite{del-et-al-2}, \cite{del-et-al-2.5}, \cite{del-et-al-3}  examine semiregular nets, quasi-regular nets and binodal networks, amongst other aspects, and indicate their appearance in material crystals. 
The $5$-fold classification of regular $3$-periodic discrete bond-node nets  is not so well-known in mathematical circles  and, moreover, a simple formal self-contained proof  does not seem to be  ready-to-hand, and so
we provide one here.
In particular we give an elementary direct construction of the most intriguing example, namely the $K_4$-crystal net (Coxeter \cite{cox-laves}, Hyde et al \cite{hyd-et-al}, Sunada \cite{sun-notices}).
We make use of the discrete classification to obtain, in Theorem \ref{t:regularmesh}, a classification of  regular meshes in dimensions $2$ and $3$ up to conformal affine isomorphism. We also indicate and comment on the correspondence between our constructs, notation and terminology  and that used by reticular chemists and a summary table of this correspondence appears at the end of Section 3.

The characterisation of regular linear meshes shows that they are parametrised by certain pairs $\N, F$ where $\N$ is a regular discrete string-node net and $F$ is a dense countable abelian subgroup of $\bR$. Moreover we show that certain \emph{strongly regular} meshes correspond to $F$ being a countable subfield. Motivated by this relationship  we also show, in Theorem \ref{thm:hull}, how one may  construct a \emph{minimal extension mesh} of any discrete line segment net by any countable field.

The regular and strongly regular linear meshes provide the most basic meshes in two and three dimensions where the space group acts with full transitivity. As we indicate in Section 3 it will be of interest to enlarge the classifications here to meshes with weaker forms of transitivity.

In Section 5 we turn to some fundamental  \emph{dynamical} considerations and initiate  a theory of rigidity and flexibility for string-node meshes.
We focus on continuous motions that preserve internodal string lengths and which satisfy a natural noncrossing condition, so that each motion is in fact a path of injective placements.
It is shown that for a square grid mesh in the unit square such motions, if smooth, necessarily have a simple \emph{laminar form} for some small finite time interval.  On the other hand we  show that the triadic kagome  mesh is rigid for continuously differentiable motions. This contrasts with the many ways in which the discrete kagome bar-joint framework is periodically flexible. We also construct highly flexible line segment meshes which, unlike the square grid meshes, admit localised finite motions. In this case, to pursue the allusion above, the motions are more akin to that of a planar fibred liquid, or liquid crystal, rather than that of an idealised fibred material.

We note that there are some similar aspects of rigidity and flexibility in  the analysis of ''continuous tensegrities" considered by Ashton \cite{ash}, and in the very recent analysis of semidiscrete surfaces in Karpenkov \cite{kar}.
Nevertheless meshes are in a quite different category, having no rigid basic components, and analysing their string-length preserving motions require new methods. 

In the final section we indicate some further directions and problem areas, including two problems which we paraphrase here as follows.
\medskip

1. In three dimensions the noncrossing motion of a grid mesh cube need not have a strictly laminar form for small time values, in the sense of being axially determined. Classify these small time motions. 
\medskip

2. The line segment mesh associated with the Sierpinski triangle is a limit of minimally rigid bar-joint frameworks. Determine if this string-node mesh is rigid with respect to noncrossing continuous string-length-preserving motions.
\medskip

We would like to thank Davide Proserpio and Egon Schulte for alerting us to articles in reticular chemistry, to the Reticular Chemistry Structure Resource, and to articles on the classification of discrete periodic nets.

\section{Nets and meshes: definitions}\label{s:linear}
We first define meshes and string-node nets  in the case of strings formed by lines or closed line segments. 

\begin{defn} (i)
A \emph{string-node net} $\N$ in the Euclidean space $\bR^d$ is a pair  $(N, S)$ of sets, whose respective elements are the nodes and strings of $\N$, with the following two properties.

(a) $S$ is a nonempty finite or countable set whose elements are lines, closed line segments or closed semi-infinite line segments in $\bR^d$, such that  collinear strings are disjoint.

(b) $N$ is a nonempty finite or countable set of points in $\bR^d$ given by the intersection points of strings. 

(ii) A \emph{mesh} is a string-node net such that for each $l\in S$ the set $l\cap N$ of nodes in $l$ is dense in $l$.

(iii) A net (resp. mesh) is a \emph{linear net} (resp. \emph{linear mesh}) if the strings are lines.
\end{defn}

We also refer to such a string-node net or mesh in a more precise manner as a \emph{line segment net} or  \emph{line segment mesh}, particularly when we consider more general \emph{smooth nets} and \emph{smooth meshes} for which the strings are curves. Also, when there is no conflict we simply refer to a line-segment string-node net, as defined above, as a net. It is these nets that concern us in Sections 2 to 4, whereas in Section 5 we consider motions in the space of smooth meshes.

In Section \ref{s:chemistry} we comment in some detail on the  constructs and terminology used here and those used by Delgado-Friedrichs et al and by Sunada. In particular a "$3$-periodic net" in reticular chemistry generally refers to the underlying translationally periodic graph of a material while at the same time certain notions (such as regularity) derive from an evident {maximum symmetry realisation}.

We define the \emph{body} $|\N|$ of a net  $\N$ to be the union of its strings and it is elementary that  there is a one-to-one correspondence $\N\to |\N|$ between the set of nets in a particular dimension and the set of their bodies. 

We remark that, although we do not do so here, in the case of the plane it is possible to consider more general nets and meshes in which strings may cross over. More precisely, strings may share points which are not nodes and there is a specification of the crossover choice, as in a knot diagram. Improper meshes of this form arise naturally, for example, from folding and pleating moves. 

Recall that the Archimedean, or semiregular, tilings of the plane are the tilings of the entire plane by regular polygons, with pairwise edge-to-edge connnections,  such that the isometric symmetries of the tiling act transitively on the vertices.
There are $11$ such tilings, $3$ of which are linear and so determine the bodies of $3$ linear string-node nets. We denote these linear nets as 
(i) $\N_{\rm tri}$, in the case of the regular triangle tiling, 
(ii) $\N_{\bZ^2}$, for the square tiling, and (iii) $\N_{\rm kag}$ for the kagome  tiling, by regular triangles and hexagons.
These are  \emph{discrete} nets in the sense that there is a positive  lower bound to the distances between nodes.

The three  semiregular
linear nets have the following scaling-inclusion properties. 
To formulate this assume that these linear nets are in a \emph{standard position} by which we mean that the $x$-axis is a string and that there are nodes at $(0,0)$ and $(1,0)$, with no intermediate nodes. It follows that
\medskip

(i) $|\N_{\rm tri}|\subseteq \frac{1}{m}|\N_{\rm tri}|$, for any positive integer $m$,
\medskip

(ii)   $|\N_{\bZ^2}|\subseteq \frac{1}{m}|\N_{\bZ^2}|$, for any positive integer $m$,
 \medskip
 
(iii)
$|\N_{\rm kag}|\subseteq \frac{1}{m}|\N_{\rm kag}|$, for any odd integer $m$.
\medskip

To see that (iii) holds it is enough to show that if $g$ is a string of $\N_{\rm kag}$ and $m$ is odd then the line $mg$ is a string. The strings with negative slope have intercepts with the $x$-axis that are odd integers, and the strings with positive slope have even integer intercepts, so these cases follow. The third case follows similarly by considering intercepts.

\medskip

\begin{center}
\begin{figure}[ht]
\centering
\includegraphics[width=6cm]{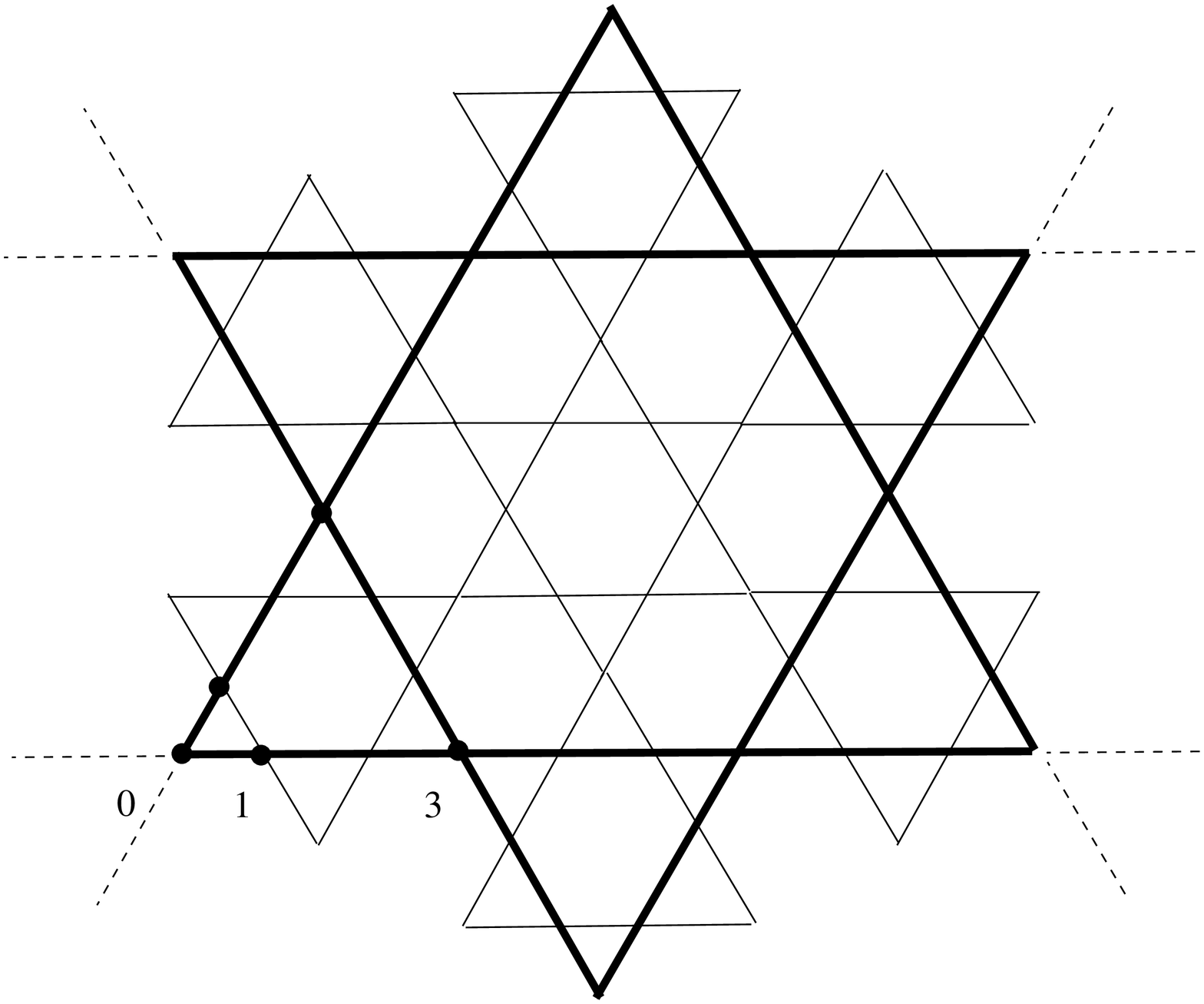}
\caption{The inclusion $3|\N_{\rm kag}|\subseteq |\N_{\rm kag}|$.}
\end{figure}
\end{center}

We define the  \textit{triadic kagome mesh} to be the linear mesh $\M_{\rm kag}$  whose body
is the union
\[
\bigcup_{n=1}^\infty \frac{1}{3^n}|\N_{\rm kag}|.
\]
This mesh is a \emph{dense mesh}, by which we mean that the nodes  are dense in the ambient space. Evidently the triadic kagome mesh is a \emph{divisible mesh} in the sense that $|\M| = a|\M|$ for some scaling factor $0<a< 1$. 

One can similarly define {dense meshes} by taking the union of scalings of
$\N_{\rm tri}$ or $\N_{\bZ^2}$. 
These examples can be considered as regular meshes, in analogy with regular tilings of the plane, which only allow a single type of regular polygon tile. These construction are possible also in three dimensions, and we  formalise a precise notion of a regular mesh in two or three dimensions in Definition \ref{d:regularlinearmesh}.

We now give some terminology related to the local structure of  nets and meshes.  

\begin{defn}\label{d:vertexfigdilngroup}
Let $\N= (N,S)$ be a (line-segment) string-node net in $\bR^d$ and let  $B(p,r)$ be the closed unit ball with radius $r$ centred at the node $p\in N$.


(i) The \emph{ray figure} $X_r(p)$ of a node $p$ is the subset of the body $|\N|$ which is the union of the line segments $l\cap B(p,r)$ with $p\in l \in S$.

(ii) The \emph{degree} or \emph{valency} of $p$, is the number of path-connected components of $X_r(p)\backslash \{p\}$, and so is a positive integer or $\infty$.

\end{defn}

We also define the \emph{vertex figure} $F(p)$ (or \emph{coordination figure})
of a node $p$ of a net $\N$ to be the closed convex hull of $rX_r(p)$ for suitably small $r$. This is well-defined for  the regular meshes defined below and agrees with the usual terminology for skew polyhedra (Coxeter \cite{cox-laves}) and for the realised periodic nets in reticular chemistry.

The next definition concerns global symmetries. 

\begin{defn}\label{d:spacegroup}
Let $\N=(N,S)$ be a  (line-segment) string-node net in $\bR^d$. For $r>0$ let $D_r$ be the scaling map on $\bR^d$ with $D_r(x) = rx$ .

(i) The \emph{space group} of $\N$ is the group $\Isom(\N)$  of isometries $T$ 
of $\bR^d$ which determine bijections $T:N \to N$ and $T:S \to S$.

(ii) The \emph{scaled space group} $\sIsom(\N)$ is the group of  maps $D_rT$, with $r>0$ and $T$ in $\Isom(\bR^d)$,  whose restriction to $N$ and to $S$ determine bijections $N \to N$ and $S \to S$.

(iii) $\N$ is \emph{periodic} if the space group contains a lattice, that is, a full rank discrete translation subgroup, determined by $d$ linearly independent period vectors.
\end{defn}


Note that the scaled isometries of $\bR^d$ coincide with the affine automorphisms of $\bR^d$ which preserve angles and so we may refer to the scaled space group as the \emph{conformal space group}.
We say that two string-node nets $\N_1=(N_1, S_1)$ and $ \N_2=(N_2,S_2)$ are  \emph{congruent} (resp.  \emph{conformally isomorphic}) if there exists an isometry $T$ (resp. scaled isometry 
$D_rT$) whose restriction to $N_1$ and $S_1$ determine bijections
$N_1 \to N_2$ and $S_1 \to S_2$.

We note some simple examples of meshes.

For countable dense sets $E_1, E_2, \dots , E_d$ in $\bR$  define $\M_{\rm grid}(E_1, E_2, \dots , E_d)$  to be  the \emph{general grid mesh} in $\bR^d$ whose body is the union of the strings with equations $x_i=a_i$, for $a_i\in E_i$.
If $d=2$ and both sets  are aperiodic then evidently 
 $\M_{\rm grid}(E_1, E_2)$ is a planar mesh with no nontrivial translational symmetries. On the other hand if $E_1=E_2$ is a countable abelian group $F$ then the associated planar mesh is rich in translational symmetries. We write this mesh as $\N_{\bZ^2} \otimes F$ and note that it is a regular mesh in the sense below.

For a contrasting example define the \emph{rational mesh} $\M(\bQ^2)$ as the linear dense mesh in $\bR^2$ whose body is the union of the lines which are \emph{rational} in the sense that
they pass through two (and hence infinitely many) points with rational coordinates. Note that  the set of nodes is $\bQ^2$ and each node has infinite degree. 

\subsection{Regular nets} The next definition  provides a formal definition of a  \emph{regular string-node net},  in two or three dimensions, and it is terminologically convenient for us to require a regular string-node net to be a discrete net. The sense of the term "regular" used here derives from the notion of a \emph{regular $3$-periodic net} considered
in reticular chemistry. In those considerations the vertex figure, or figures, of such  nets  appear prominently in various classification schemes. See, in particular,
Delgado-Friedrichs and O'Keeffe \cite{del-oke},  Delgado-Friedrichs,  O'Keeffe and  Yaghi \cite{del-et-al-1}, \cite{del-et-al-2},
Delgado Friedrichs et al \cite{del-et-al-3} and our comments in Section \ref{s:chemistry} below.

Recall that a countable subset of $\bR^d$ is said to be \emph{relatively dense} if there is a radius $r$ such that each point of $\bR^d$ lies in a closed ball $B(p,r)$ of radius $r$ centred at $p$, for some point $p$ in the set.

\begin{defn}\label{d:regulardiscretenet}Let $d=2$ or $3$. A \emph{regular string-node net} 
$\N$ in $\bR^d$ is a connected discrete {periodic} line segment net in $\bR^d$ with the following properties.

(i) $\N$ is \emph{transitive}, in the sense that for every pair of nodes $p_1, p_2$ there exists
$T\in \Isom(\N)$ with $T(p_1)=p_2$.

(ii) The vertex figure $F(p)$ of any node is a regular polygon or a regular polyhedron centred at $p$.

(iii) For some (and hence every) node $p$  the rotational isometries of $\bR^d$ which fix $p$ and induce symmetries of $F(p)$ are contained  in the space group of $\N$.

(iv) The nodes are \emph{relatively dense} in $\bR^d$.
\end{defn}
 
Note that  condition (ii) allows the possibility of a regular polygon vertex figure in three dimensions, rather than requiring a regular polyhedron vertex figure. Condition (iv) is in fact redundant, being a consequence of our full rank notion of periodic, but we include it to maintain a parallel with the definition of a regular mesh below.

In two dimensions it is readily proven that there are $3$ regular nets up to conformal isomorphism, namely $\N_{\rm tri}$ and  $\N_{\bZ^2}$, which are linear, and the net  which derives from the hexagonal honeycomb tiling which we denote as $\N_{\rm hex}$.
In three dimensions, much less evidently, there are $5$ regular nets, identified by Delgado-Friedrichs,  O'Keeffe and  Yaghi in \cite{del-et-al-1}. 
We give a formal statement and proof of this fact in Theorem \ref{t:regularnets}.

\medskip


\subsection{Regular meshes}In the following definition of a \emph{regular linear mesh} we have the conditions corresponding to the requirements (i), (ii), (iii) of a regular net, while the relative density condition (iv) is replaced by density.

\begin{defn}\label{d:regularlinearmesh}
Let $\M=(N,S)$ be a {periodic} (line-segment) mesh in $\bR^d$ for $d=2,3$. Then $\M$ is  \emph{regular}  if the following properties hold.

(i) $\M$ is \emph{transitive}, in the sense that for every pair of nodes $p_1, p_2$ there exists
$T\in \Isom(\M)$ with $T(p_1)=p_2$.

(ii) The convex hull of each ray figure $X_r(p)$ is a regular polygon or a regular polyhedron centred at $p$.

(iii) For some (and hence every) node $p$  the rotational isometries of $\bR^d$ which fix $p$ and induce symmetries of $X_r(p)$ are contained  in the space group of $\M$.

(iv) The nodes are dense in $\bR^d$.
\end{defn}

Note that for a regular mesh the isometries that fix a given node are finite in number. It follows readily from this and the transitivity and density properties that a regular line-segment mesh is necessarily a linear mesh.

The transitivity property (i)  is a natural form of homogeneity making the position of each node in the net of equal status. For regular nets meshes it is also natural to consider the following stronger form of homogeneity.

\begin{defn}Let $\N$ be a string-node  net in $\bR^d$.

(i) $\N$ is \emph{strongly transitive} if for every pair of pairs of nodes,  $\{p_1, p_1'\}, \{p_2, p_2'\}$, with each pair lying on the same string and with $p_1\neq p_2$ and $p_1'\neq p_2'$ there exists a scaled isometry
$T$ in $\sIsom(\M)$ with $T(p_1)=p_2$ and $T(p_1')=p_2'$.

(ii) $\N$ is a \emph{strongly regular linear mesh} if it is a regular mesh and is strongly transitive.
\end{defn}

The following scaling group invariant in the case of meshes in standard position, plays a useful role in various constructions and is nontrivial for strongly transitive meshes. On the other hand this group is trivial for any discrete net and is trivial  for the grid mesh $\M_{grid}(E,E)$ where $E$ is the additive group of rationals $a/b$
where $b$ is any product of distinct primes.

\begin{defn}
 The \emph{scaling group}, or dilation group,  of a net $\N$, in $\bR^d$, with a node at the origin, is the group $D(\N)$ of positive real numbers $r$ for which $r|\N|  = |\N|$, with multiplicative group product. 
\end{defn}

\subsection{Smooth meshes}We recall that a continuous flex of a bar-joint framework $(G,p)$, with placement vector $p \in \bR^{d|V|}$,  is a continuous map $p: [0,1] \to \bR^{d|V|}$ for which the
frameworks $(G,p(t))$ are equivalent to $(G,p(0))$, where $p(0)=p$.
We consider later how a linear mesh or a line segment mesh may similarly flex through a set of more general meshes, subject to the preservation of internodal distances on strings.
The string-node meshes that play this role may be defined as follows.

\begin{defn}
A \textit{$C_k$-smooth string} $L$  in $\bR^d$ is a closed set which is the range of a $k$-fold continuously differentiable injective map $\gamma :I\to \bR^d$, where $I$ is one of the intervals $[0,1], [0,\infty),$ or
$(-\infty, \infty)$.
\end{defn}

\begin{defn}\label{d:net} A \emph{$C_k$-smooth string-node net} in $\bR^d$ 
is a pair $\N=(N, S)$ where $S$ is a countable set of $C_k$-smooth strings $l$ in $\bR^d$ whose pairwise intersections are finite in number and where $N$ is the set of points which lie in two or more strings, and is a nonempty set.
\end{defn}

\begin{defn}
A \emph{$C_k$-smooth mesh} 
is a $C_k$-smooth string-node net $\M=(N, S)$ in $\bR^d$ whose set of nodes, $N$, is dense in the body of $\M$. 
\end{defn}

For a $C_1$-smooth string one may define the length between any two of its points. When these lengths are finite the string  may be parametrised by its oriented length from some initial point $x$. Thus, in this case there exists
a reparametrisation $\gamma :I\to \bR^d$ so that the length of the simple path $\gamma([a,b])$ is $b-a$ for all $a\leq b$.

For a simple  example of a smooth mesh (that is, a $C_\infty$-smooth mesh) consider the two-dimensional \textit{radial mesh} $\M_{\rm radial}$
whose set of strings consists 
of concentric circles about the origin with rational radii
together with the straight lines through the origin whose angles
with the $x$-axis are rational multiples of $\pi$.
Note that the central node has infinite degree while all other nodes have degree $4$. 

\section{Regular discrete nets}In this section
we classify the regular nets of Definition \ref{d:regulardiscretenet} in the case of three dimensions.
\medskip

(I) There is a regular net in $\bR^3$, which we denote as $\N_{K_4}$, with vertex figure equal to an equilateral triangle. 
\medskip

Coxeter \cite{cox-PLMS} notes the early discoveries of this net by the mathematician John Petrie and by the crystallographer Fritz Laves, and refers to it as the Laves graph. There are in fact two chiral forms of the net, in the usual sense that its mirror image 
lies in a distinct rotation-translation class.
It is known by various names, associated in part with its appearance in crystal structures and we comment further on this later in this  section. In the terminology of Toshikazu Sunada it is referred to as the $K_4$-crystal, reflecting the fact that it is the so-called standard realisation of the abelian covering graph of the graph $K_4$. Details of its construction in this manner
are given in Sunada \cite{sun-notices}, \cite{sun-book}. In \cite{cox-laves} Coxeter gives a group theoretic construction of it as a graph inscribed on an infinite regular skew polyhedron.
In contrast we give a direct geometric construction below together with  an explicit inclusion sequence which helps clarify its construction, namely
\[
\N_{K_4} \subset \N_{\rm Hxg} \subset \N_{\rm Fcu}
\]
\medskip

(II) There is a regular net, denoted $\N_{\rm Scaff}$, with vertex figure equal to a square. 
\medskip

We refer to this string-node net as the scaffolding net since finite parts of it resemble a scaffolding structure. 
Figure \ref{f:scaff} indicates its construction by means of a specification of $6$ representative nodes and $6$ representative strings for the translation classes of the set of nodes and strings with respect to the periodicity vectors
$(2,0,0), (0,2,0), (0,0,2)$.
Thus it may be considered as a periodic string and node depletion of a containing grid net. The \emph{translation classes} referred to here and elsewhere are taken with respect to the periodicity vectors associated with a specified  unit cell or translation group.

\begin{center}
\begin{figure}[ht]
\centering
\includegraphics[width=5cm]{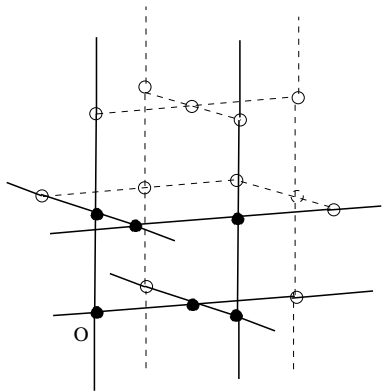}
\caption{A generating set of $6$ nodes and $6$ linear strings for  $\N_{\rm Scaff}$.}
\label{f:scaff}
\end{figure}
\end{center}
\medskip

(III) There is a regular net, denoted $\N_{\rm Dia}$, with vertex figure equal to a tetrahedron.
\medskip

This net corresponds to the well-known bond-node structure of the diamond form of carbon. 
Geometric detail for this net  can be found in \cite{pow-poly} and \cite{sun-book}, for example. 
\medskip

(IV) There is a regular net, denoted $\N_{\rm Bcu}$, with vertex figure equal to a cube.
\medskip

The set of nodes of this net may be defined as an augmentation of the set ${\bZ^3}$  by  nodes at the centre of each of the associated cubical cells. The set of strings  is the set of lines which pass through such a central node and a nearest neighbour. 
Our notation reflects the fact that this string node net is associated with the body centred cubic (bcu)  lattice in the Bravais lattice division of crystal types.
\medskip

(V) There is a regular net, denoted $\N_{\bZ^3}$, with vertex figure equal to an octahedron.
\medskip

This is the linear net $(N,S)$  with node set  $N=\bZ^3$ and with $S$ the set of lines given by integer translates of the coordinate axes.

\begin{thm}\label{t:regularnets}
In three dimensions there are five regular string-node nets up to conformal isomorphism, namely
$\N_{K_4}, \N_{\rm Scaff}, \N_{\rm Dia}, \N_{\rm Bcu}$ and $\N_{\bZ^3}$.
\end{thm}

\begin{proof}We first construct the net $\N_{K_4}$. In this argument we identify an interesting transitive connected linear periodic net, which we denote as $\N_{\rm Hxg}$, whose coordination figure at every vertex is a regular hexagon but which, nevertheless, falls short of the symmetry-rich condition (iii) in Definition \ref{d:regulardiscretenet}.

Recall the well-known crystallographic restriction for the rotational symmetries of any full rank periodic net in three dimensions. This ensures that such  symmetries have order $1, 2, 3, 4,$ or $6$. In view of this, in addition to the five vertex figure types appearing in (I), \dots , (V) (ie., regular triangle, square, regular tetrahedron, cube, regular octahedron) there is just one additional vertex figure case to be considered,  namely a regular hexagon.  

To define $\N_{\rm Hxg}$ define first the \emph{face-centred cubic net} $\N_{\rm Fcu}$ as the linear net whose strings are the straight lines which pass through the midpoints of opposite edges of cubes that appear in a regular cubical partition of $\bR^3$. (The companion net $\N_{\rm Bcu}$ has straight line  strings which pass through opposite corners of the cubes in such a division.)
The vertex figure of $\N_{\rm Fcu}$ is  a cuboctahedron, a semiregular polyhedron. 

Figure \ref{f:fcc} indicates a "unit cell" (dotted) for $\N_{\rm Fcu}$ with respect to 
the orthogonal periodicity vectors $(2,0,0), (0,2,0), (0,0,2)$, together with the following additional features.
\medskip

(a) The solid edges and the dashed edges  indicate the position of line segments which together provide $24$ representatives for the $24$ translation classes of the internodal edges of $\N_{\rm Fcu}$.
\medskip

(b) The node labelled $B$ (considered a blue node) at the centre of the cell, together with the nodes labelled $g, r, v$ (considered as coloured green, red and violet), provide $4$ representatives for the translation classes of the nodes.
\medskip

The \emph{regular hexagonally coordinated net} in three dimensions is defined to be the string-node net, $\N_{\rm Hxg}$, determined by the $12$ solid edges and their translation classes. In the terminology of \cite{pow-poly} the finite sets in (a) and (b) provide a \emph{motif} for the associated periodic bar-joint framework. (In the terminology of Sunada \cite{sun-book} the solid edges determine the vectors of a \emph{building block} for $\N_{\rm Hxg}$.)

\begin{center}
\begin{figure}[ht]
\centering
\includegraphics[width=5.5cm]{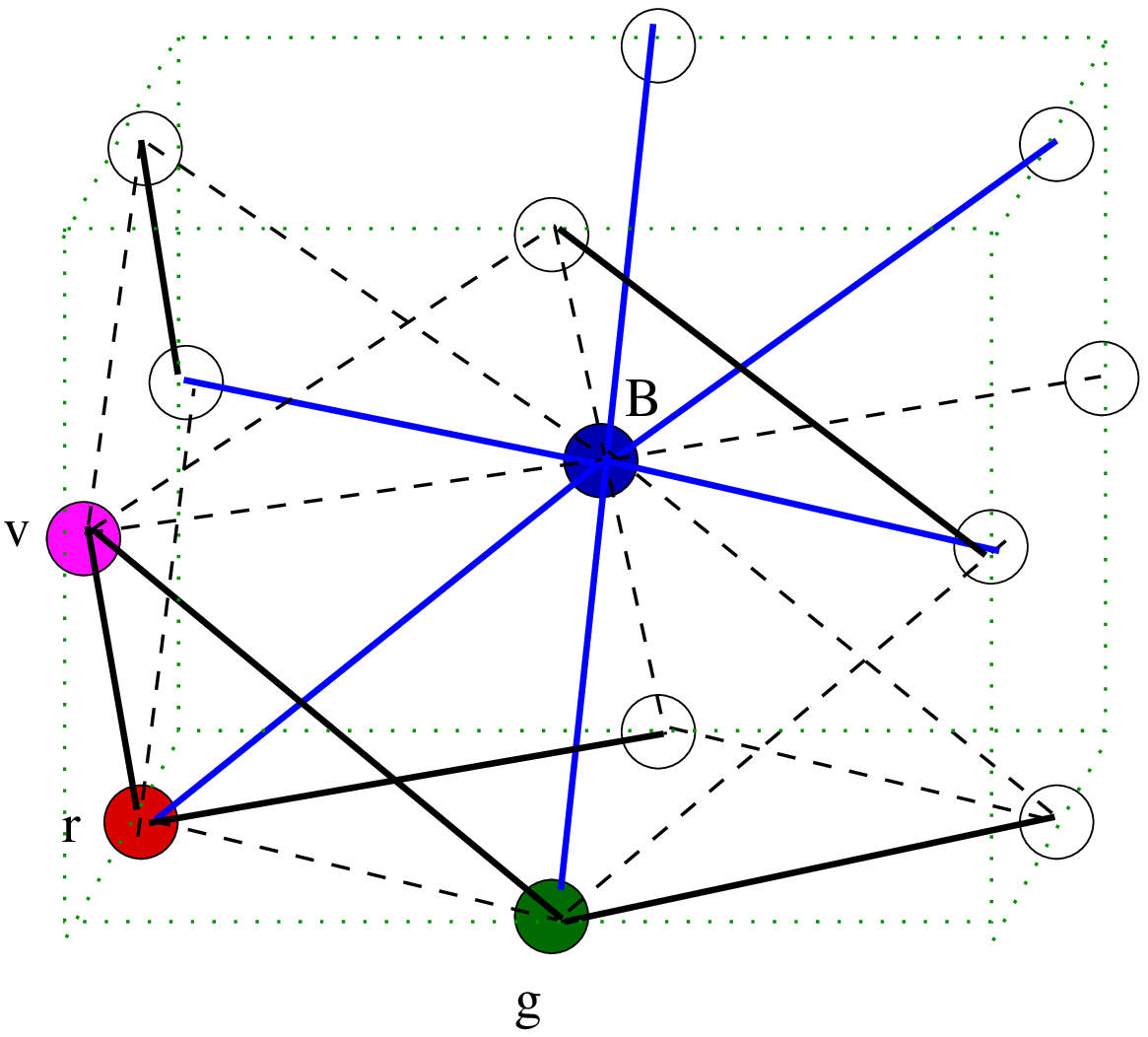}
\caption{The identification of $\N_{\rm Hxg}$ in $\N_{\rm Fcu}$.}
\label{f:fcc}
\end{figure}
\end{center}

\begin{center}
\begin{figure}[ht]
\centering
\includegraphics[width=8cm]{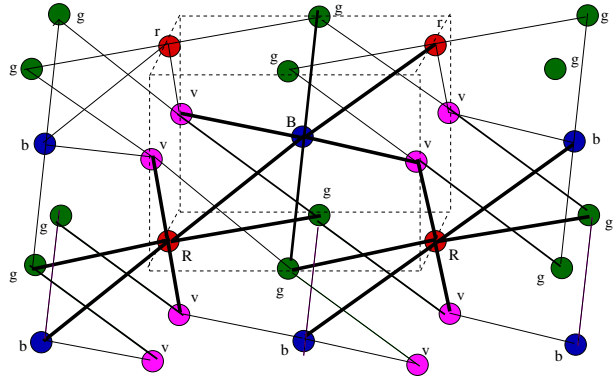}
\caption{Part of the linear net $\N_{\rm Hxg}$.}
\label{f:Hxg}
\end{figure}
\end{center}

Note that the vertex figure of the central node $B$ is  a regular hexagon whose containing plane has a normal vector $(1,-1,1)$. Moreover one can readily check that vertex figures
of the red, violet and green nodes, as nodes of $\N_{\rm Hxg}$, are also hexagons, with normal vectors $(-1,-1,1), (1,1,1)$ and $(-1,1,1)$ respectively. Thus, there are $4$ translation classes for the ray figures of $\N_{\rm Hxg}$.

The net $\N_{\rm Hxg}$ is a linear net
which we may view as a periodic linear string depletion of  $\N_{\rm Fcu}$, in analogy with the relationship between $\N_{\rm Scaff}$ and $\N_{\bZ^3}$.
Figure \ref{f:Hxg} shows a rectangular block portion of it which includes one interior blue node (labelled $B$) and two interior red nodes (labelled $R$) and their ray figures. 

Note that the ray figure of the central  $B$-labelled node in Figure \ref{f:Hxg} has a rotational symmetry of order $6$ which maps the two red nodes to the two green nodes. This symmetry does not extend to a symmetry of  $\N_{\rm Hxg}$ since the edges from  the two red nodes and violet nodes do not rotate to edges incident to the green nodes.
Thus $\N_{\rm Hxg}$ is not a regular net. (The transitivity property holds, with the equation $Tp_1=p_2$ being realised by a translation isometry, if the nodes $p_1$ and $p_2$ have the same colour, and by an appropriate translation-rotation isometry if they are of different colour.) 
\begin{center}
\begin{figure}[ht]
\centering
\includegraphics[width=8cm]{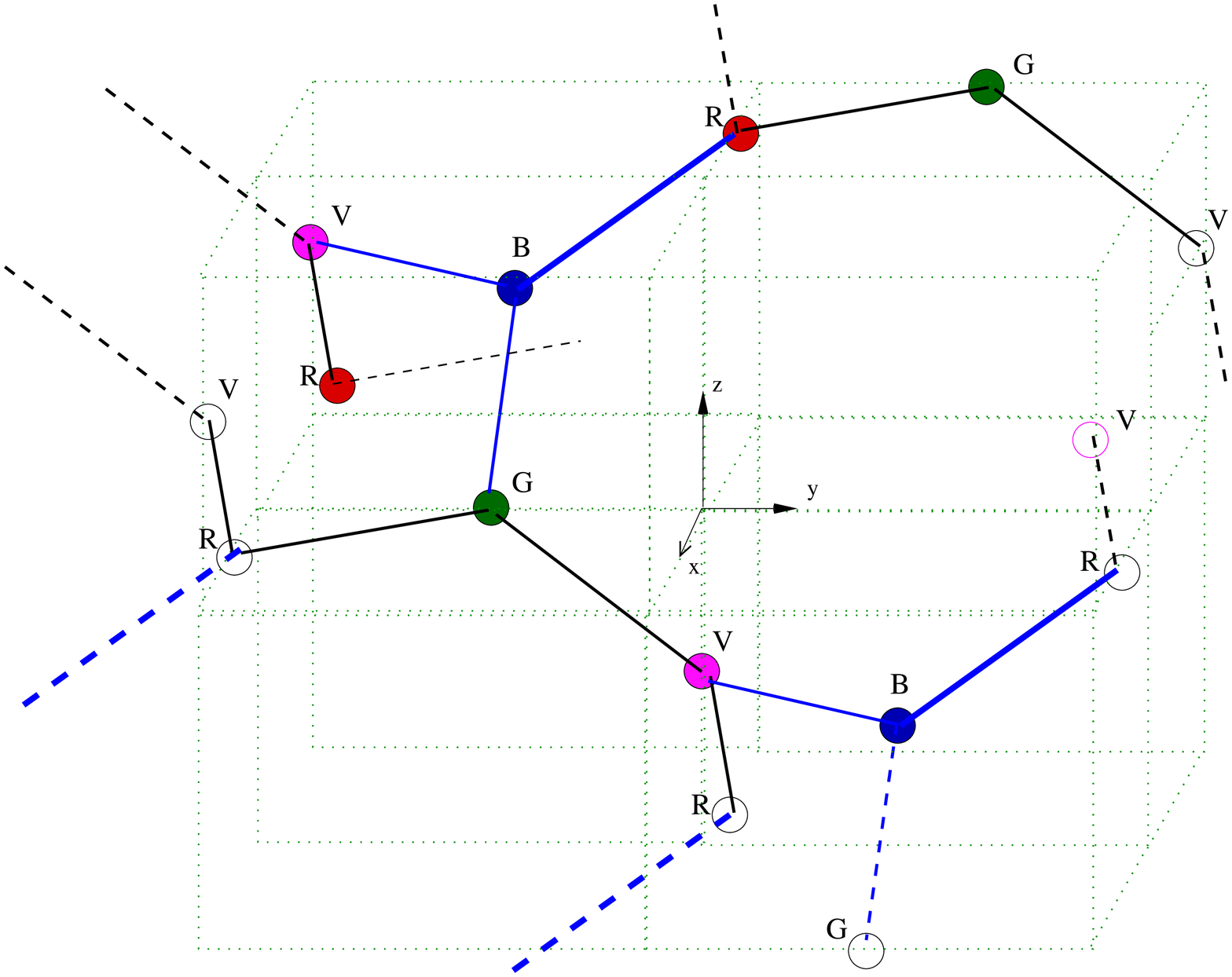}
\caption{The identification of  $\N_{K_4}$ in $\N_{\rm Hxg}$ and $\N_{\rm Fcu}$.}
\label{f:K4cell}
\end{figure}
\end{center}

The construction of $\N_{\rm Hxg}$ above is given by specifying subsets in a particular (nonprimitive) unit cell for the face-centred cubic net. 
In an entirely similar way we now define $\N_{K_4} \subset \N_{\rm Hxg}$ by specifying  limited sets of nodes and edges in a new larger unit cell for $\N_{\rm Hxg}$, with periodicity vectors $(4,0,0), (0,4,0), (0,0,4)$.  The result of this is indicated in Figure \ref{f:K4cell} by the $8$ solid nodes and the $12$  solid internode edges. We start the specification  by including the  blue node, labelled  $B$, with coordinates $(-1,-1,1)$. A choice is then made of one of the two triples of edges incident to this node which provides a triangular ray figure. The three edges each determine, uniquely, such an edge triple for the neighbouring nodes, labelled $G, R$ and $V$. This determination propagates throughout the unit cell $[-2,2]\times [-2,2]\times [-2,2]$ to define $12$ edges. (See Figure  \ref{f:K4cell}.) Moreover, since the node colouring is periodic this "common edge" determination propagates periodically beyond the unit cell and so determines a connected discrete periodic net, which we define to be $\N_{K_4}$.

Note that the subgraphs  determined by two ray figures that share an edge are mutually congruent by translation-rotation isometries, and
that this double ray figure graph is chiral. It follows that $\N_{K_4}$ is also chiral. Also it follows that
$\N_{K_4}$ is invariant under the rotations of $\bR^3$ of order $2$ and of order $3$ that act on a ray figure. Thus condition (iii) in Definition \ref{d:regulardiscretenet} is satisfied.
These rotations permute the $4$ colourings of nodes in a transitive
manner and so condition (iv) also follows. Thus $\N_{K_4}$ is a regular net.

To complete the proof of Theorem \ref{t:regularnets} it remains to show that each of the five regular nets is uniquely determined as a regular net, up to conformal isomorphism, by its vertex figure, and that $\N_{\rm Hxg}$ is also uniquely determined as a transitive  periodic net, with vertex figure a regular hexagon, satisfying conditions (i) and (iv) of Definition \ref{d:regulardiscretenet}.

This uniqueness is elementary when the vertex figure is a square (as with $\N_{\rm Scaff})$, a cube ($\N_{\rm Bcu}$) or an octahedron ($\N_{\bZ^3}$). It remains to consider $\N_{K_4}$, $\N_{\rm Hxg}$ and  $\N_{\rm Dia}$.

Let $\N$ be a periodic transitive connected string-node net with vertex figure equal to a regular triangle and
suppose that the symmetry-rich condition (iii) of Definition \ref{d:regulardiscretenet} holds.  Note that $\N$ is not planar, by the periodicity condition, and so the normals of the ray figures of $\N$ are not parallel. By the condition (iii) the normals through the central nodes of the ray figures are axes for $3$-fold rotations. It follows that $\N$ has at least $2$ distinct axial directions for $3$-fold rotation and since $\N$ is periodic  there exist an orthogonal triple of period vectors
and the rotation axes have directions corresponding to the diagonals of a cube. It now follows that a double ray figure for $\N$ is congruent to a double ray figure for $\N_{K_4}$. It also follows, by the rotational symmetries, that the double ray figures of $\N$ have the same chirality. Thus if we translate-rotate $\N$ in order to identify a single ray figure of $\N$ with a ray figure of $\N_{K_4}$, then $\N$ is either equal to $\N_{K_4}$ or its mirror image. 
The arguments for $\N_{\rm Hxg}$ and  $\N_{\rm Dia}$ are  similar. 
\end{proof}

\subsection{The $K_4$-crystal} Note that the structure of the building block motif in Figure \ref{f:K4cell} implies the dominant helical features that are visible in larger fragments of $\N_{K_4}$. In fact there are two types of  infinite helical paths in the direction of the positive $x$-axis which are determined by the periodic colour sequences 
\[
...(VGRBGVBR)(VGRBGVBR)... \quad \mbox{ and } \quad ... (RVBG)(RVBG)...
\]  
Also, by the rotational symmetries of order $3$, there are similar helical paths in the direction of the $y$-axis and $z$-axis. 

It is also possible to see that the mirror image of $\N_{K_4}$ in the plane $x=0$ shifted $2$ units in the $y$ direction is disjoint from $\N_{K_4}$. These two infinite graphs "map out" the two chambers of a infinite minimal surface discovered by A. H. Schoen. For more mathematical commentary on the $K_4$ net see Coxeter \cite{cox-laves} and Sunada \cite{sun-book}.
Also, see Hyde et al \cite{hyd-et-al} and Schr\"{o}der-Turk et al \cite{sch-et-al} for an intriguing  indication of its occurrences in so-called soft materials in chemistry, and in the wing structure of a butterfly.

\subsection{Quasiregular nets and semiregular nets}  We define a  \emph{quasiregular} string-node net in three dimensions to be a discrete connected {periodic} net which satisfies the conditions (i) to (iv) of Definition \ref{d:regulardiscretenet} but with condition (ii) replaced the requirement that the vertex figure is a \emph{semi-regular} (Archimedean) polyhedron which is not a regular polyhedron. It can be proven, as above, that there is only one such net in three dimension namely $\N_{\rm Fcu}$, with the cuboctahedron as vertex figure. 

We also define a \emph{semiregular string-node net} in the following manner (echoing Delgado-Friedrichs et al \cite{del-et-al-1},\cite{del-et-al-2})
by making reference to a hierarchy of transitivity. Specifically,
\emph{$pqrs$-transitivity} means that there are $p$ types of vertex, $q$ types of edges,  $r$ types of face and $s$ types of tile, up to isometric symmetries. The regular string-node nets in $\bR^3$ may then be defined as the connected (full rank) periodic discrete line segment nets with transitivity type $1111$. Also $\N_{\rm Fcu}$ has type $1112$ and is the unique net of this type, while $\N_{\rm Hxg}$ has type $1121$. 
We define a \emph{semiregular} string-node net to be a  (full rank) periodic discrete net with transitivity type $11rs$ with $r\geq 2$.  This more general mode of definition requires an appropriate definition of faces and tiles which we need not expand on here. In \cite{del-et-al-2} $14$ semiregular periodic nets are identified  which also have the (sphere packing) property that there are no intervertex distances less than the common edge length (in a
maximum-symmetry embedding). See also the detailed discussion in Pellicer and Schulte \cite{pel-sch} and Schulte \cite{schulte}.

\subsection{Terminology and notation}\label{s:chemistry}We note that a "$3$-periodic net" in reticular chemistry refers to a countable graph which underlies a periodic bond-node structure in $\bR^3$. At the same time certain notions for these nets, such as regularity, derive from a reference to a form of "the most symmetrical embedding" in $\bR^3$. In a more mathematical vein, Sunada defines such a graph, with the name \emph{topological crystal}, as an infinite-fold regular covering graph over a finite (not necessarily simple) graph.  Constructions are then considered which yield various standard realisations or equilibrium realisations and so to implicit enumeration  of geometric periodic bond-node structures in $\bR^3$, that is, to mathematical crystals.

On the other hand the periodic string-node nets considered here are more simple-minded primary geometric constructs, logically analogous to material crystal perhaps, and the underlying graphs, in the discrete case, and associated invariants are subsequent considerations. 

The notation we have used for string-node nets and meshes echoes that used for crystallographic bar-joint frameworks in Power \cite{pow-poly}. There we employed a mainly mnemonic notational style with, typically, $\C_{\rm xyz}$ indicating a planar periodic bar-joint framework, and $\C_{Xyx}$, with leading letter upper case, indicating a crystallographic framework in three dimensions. Additionally,  $C_{\rm XYZ}$ indicated a three-dimensional bar-joint framework (eg. $\C_{\rm SOD}$) which is determined in a well-defined way from a zeolite framework with name XYZ (eg. SOD, for the cubic form of sodalite).  The notation  $\N_{\rm Xyz}$ readily allows adornments to indicate derived structures, such as $\N_{\rm Bcu}\otimes \bQ$.

The following summary table, for the discrete periodic line segment
string-node nets we have considered above, includes the correspondence with the chemistry-derived terminology ({\bf srs} etc) that has been given by Delgado-Friedrichs, O'Keeffe and Yaghi.
This represents only the tip of the iceberg with regard to the striking array of periodic nets that are considered by chemists and co-workers who are interested in mapping the $pqrs$-transitivity territory of the uninodal nets that underly material crystals. See in particular the Reticular Chemistry Structure Resource \cite{rcsr} at http://rcsr.net/. In the absence of any compelling mnemonic or mathematical reason it may well be natural to carry over the chemical name {\bf abc} of a net when it gives an unequivocal string-node net, $\N_{\bf abc}$ say, which has  well-defined  node coordinates in $\bR^3$, up to conformal isomorphism.

\medskip

\begin{center}
\begin{tabular}{|c|c|c|c|c|c|}
\hline
$\N$& valency & figure & {\bf abc}   & comment  \\
\hline
\hline
$\N_{\rm hex}$&$3$& regular triangle  & {\bf hcb}& graphene \\
$\N_{\bZ^2}$&$4$& square & {\bf sql}  & 2D grid\\
$\N_{\rm tri}$&$6$& hexagon & {\bf hxl} & triangle grid\\
\hline
$\N_{\rm kag}$&$4$& rectangle & {\bf kgm} & 2D kagome \\
\hline
\hline
$\N_{K_4}$& $3$& reg. triangle& {\bf srs}  & SrSi$_2$\\
$\N_{\rm Scaff}$& $4$& square & {\bf nbo}  &NbO\\
$\N_{\rm Dia}$& $4$& reg. tetrahedron& {\bf dia}  &diamond\\
$\N_{\bZ^3}$& $6$& reg. octahedron & {\bf pcu} & primitive cubic\\
$\N_{\rm Bcu}$& $8$& cube& {\bf bcu} & body centred cubic\\
\hline
$\N_{\rm Fcu}$& $12$& cuboctahedron& {\bf fcu}  &face-centred cubic\\
\hline
$\N_{\rm Hxg}$& $6$& hexagon& {\bf hxg}  & - \\
\hline
\end{tabular}\\
\end{center}
\medskip

\section{Regular  meshes in $\bR^2$ and $\bR^3$}\label{s:linear}
In this section we classify regular line meshes. As we have already observed the line segment meshes which are regular are necessarily linear. The proof makes use of the following theorem which is a consequence of Theorem \ref{t:regularnets}. However there is also a more direct proof since there is no need to consider the nonlinear discrete nets $\N_{K_4}$ and $\N_{\rm Dia}$. 

\begin{thm}\label{t:linearregularnets} Up to conformal isomorphism there are $3$ linear regular string-node nets in three dimensions, namely $\N_{\rm Scaff}, \N_{\rm Bcu}$ and
$\N_{\bZ^3}$.
\end{thm}

Consider the standard realisation of the regular linear net $\N_{\rm tri} = (N,S)$ with a node at the origin and period vectors $a_1= (1,0), a_2= (1/2,\sqrt{3}/2))$.
Let $F$ be a countable additive subgroup of $\bR$ with $1 \in F$ and define $\N_{\rm tri}\otimes F$ to be the  linear mesh whose strings are the translates
\[
\lambda_1a_1+\lambda _2a_2 + l, \quad \lambda_1, \lambda_2 \in F\cap [0,1),\quad l \in S.
\]
Then it follows  that $\N_{\rm tri}\otimes F$ is a regular linear mesh whose vertex figure is a hexagon. 

More generally, let $\N$ be
a regular linear net in $\bR^d$, for $d=2$ or $3$, which is in standard position. Recall that such nets by our definition are discrete and we may assume that the internodal distance are equal to $1$. By a \emph{standard position} for a linear 3D mesh we mean that the origin is a node, the $x$-axis is a string, and that there is a string through the origin in the $x,y$-plane which has minimal angle with the $x$-axis. Then we may define $\N\otimes F$ in the same way as above but in terms of a triple of period vectors.

In two dimensions there is only one other possibility for a regular linear net $\N$, namely $\N_{\bZ^2}$, and one can check that $\N_{\bZ^2}\otimes F= \M_{\rm grid}(F,F)$, a regular linear mesh with vertex figure equal to a square. 
Also it follows that for the regular linear meshes $\N_{\rm tri}\otimes F$ and  $\N_{\bZ^2}\otimes F$ the scaling group is the group of order preserving abelian group automorphisms of $F$ and
hence that the scaling group coincides with $F_*=F\cap \bR_+$ if and only if $F$ is a field.

In three dimensions one obtains similar conclusions for the meshes $\N_{\rm Bcu}\otimes F$ and 
$\N_{\bZ^3}\otimes F$. However, for the scaffolding net $\N_{\rm Scaff}$ and certain choices of
$F$ we may obtain the identity $\N_{\rm Scaff}\otimes F = \N_{\bZ^3}\otimes F$.
The reason for this, roughly speaking, is that if the additive subgroup $F$ contains $1/2$ then the translates of the strings will put back missing strings which distinguish $\N_{\rm Scaff}$ from $\N_{\bZ^3}$. 

Let us say that the additive group $F$ (which contains $1$) is \emph{an even subgroup} if $1/2^k \in F$ for all $k$, and that $F'$ is \emph{an odd subgroup}  otherwise.

We have now identified various regular linear meshes in $\bR^2$ and $\bR^3$ which we denote as follows, where $F$ (resp. $F'$) is an additive countable subgroup (resp. odd subgroup) of $\bR$ containing $1$.
\medskip

\begin{itemize}
\item[]
\noindent $\dim 2$: \quad
$\M_{\rm tri}(F) = \N_{\rm tri}\otimes F,  \M_{\bZ^2}(F)= \N_{\bZ^2}\otimes F$
\medskip

\item[]
\noindent $\dim 3$: \quad
$\M_{\rm Scaff}(F') = \N_{\rm Scaff}\otimes F',
  \M_{Bcu}(F) = \N_{Bcu}\otimes F,  \M_{\bZ^3}(F)= \N_{\bZ^3}\otimes F$
\end{itemize}
\medskip

\begin{thm}\label{t:regularmesh}
Let $\M$ be a regular linear mesh in $\bR^2$ or $\bR^3$. Then $\M$ has one of the forms above, up to conformal isomorphism. Moreover $\M$ is strongly regular if and only if the associated additive group $F$, or $F'$, is a field, in which case the dilation group is $F_*$, or $F'_*$.
\end{thm}

\begin{proof}
Let $\M$ be a regular linear mesh in $\bR^2$ or $\bR^3$ which is in a standard position. Then the nodes on the linear strings through the origin node have nodes with string distance $1$ from the origin which 
are positioned at the vertices of a regular polygon or polyhedron $\P$ whose centroid is at the origin. By the transitivity condition (i) and the local symmetry condition (iii) of Definition \ref{d:regularlinearmesh} it follows that $\M$ contains as a subnet a discrete regular (linear string-node) net $\N$, in a standard position, whose vertex figure coincides with $\P$. Thus for $d=2$ there are two possibilities for $\P$ and for $d=3$ there are $3$ possibilities.

Let $F$ be the subset of $\bR$ for the  positions of the nodes of $\M$ on the $x$-axis. By the conditions (i) and (iv) of Definition \ref{d:regularlinearmesh} it follows that $F$ is a countable dense additive abelian subgroup of $\bR$ with $1 \in F$. By conditions (i) and (iii) again it follows that $\M$ contains the mesh $\N \otimes F$ as a submesh.
Since this submesh contains all the strings and nodes of $\M$ it coincides with  $\N \otimes F$. 

Suppose that $\M$ is strongly regular. Then in addition to the above
we have the property that if $p_1, p_2$ and $p_3, p_4$ are pairs of nodes, with $p_1, p_2$ and  $p_3, p_4$ distinct, and with each pair on the same string, then there exists a scaled isometry $T$ of $\M$ with $Tp_1=p_3, Tp_2=p_4$.
Let $p_1, p_3$ be the node at the origin and let $p_2$ be on the $x$-axis with coordinate $x_1$, and let $p_4=1\in F$. Since $T$ fixes the $x$-axis and the origin it is a linear map and so the restriction to the $x$-axis  is multiplication by $1/x_1$. Thus $1/x_1=T(1) \in F$ and it follows that $F$ is a field.

On the other hand suppose that $F$ is a field and that
$p_1, p_2$ and  $p_1', p_2'$ are distinct points taken from $2$ strings. Since $\M$ is regular we may assume, in proving the strong regularity condition, that $p_1=p_1'=0$ and that $p_2$ lies on the $x$-axis.
Since $F$ is a field the scaled isometry group of $\M$ contains the group of dilation maps $x \to \alpha x$, where $x \in \bR^d, \alpha\in F_*$.
In particular we may assume that $|p_2|=|p_1-p_2|=|p_1'-p_2'|=|p_2'|=1,$  so that  $p_1$ and $p_2$ are vertices of the vertex figure for the origin. Since $\M$ is regular there is a linear isometry that maps $p_2$ to $p_2'$, and this completes the proof.
\end{proof}

In a similar manner one may define directly meshes based on the kagome net, namely 
\[
\M_{\rm kag}(F')= \N_{\rm kag}\otimes F'
\]
where $F'$ is an odd countable additive subgroup of $\bR$. These are precisely the nets which satisfy the conditions (i), (iii), and (iv), and have the property that the vertex figures is equal to the nonregular rectangular vertex figure of $\N_{\rm kag}$.

We now look more closely at regular meshes for which
$F$ is a unital subgroup of $ \bQ$.

Define  the \emph{supernatural number}, or \emph{generalised integer}, $n$ as the formal infinite product $n = 2^{r_1}3^{r_2}5^{r_3}\dots $, or, equivalently, as the sequence of exponents $r_1, r_2,... $,
with $r_k$ a nonnegative integer or the symbol $\infty$ for each $k$.
Two such numbers $n, m$  are  \textit{finitely equivalent}  if there are positive integers $a, b$ such that there is an equality of the formal products for $an $ and $bm$, in which case we write $an=bm$ to indicate this.

For the supernatural number $n$, as above, let $n_k, k = 1,2,\dots ,$ be any increasing sequence of natural numbers for which  $n_k$ divides $n_{k+1}$, for all $k$, and the multiplicity $r(j,k)$ of the prime divisor $p_j$ of $n_k$ tends to $r_j$ as $k\to \infty$. 
Define the proper linear mesh $\M_{\rm tri}(n)$ as the mesh $(N, S)$ with body equal to
\[
\cup_{k=1}^\infty \frac{1}{n_k}|\N_{\rm tri}|.
\]
This is well-defined in the sense that the set of strings, and hence nodes, is independent of the choice of the sequence $(n_k)$ for $n$.
In fact $\M_{\rm tri}(n)=  \N_{\rm tri}\otimes \bQ(n)$ where $\bQ(n)$ is the abelian group of positive rationals $c/d$ with $d$ a divisor of $n$.
Note that $\bQ(n)$ is a field if and only if $F=\bQ$ (corresponding to $r_k=\infty$ for all $k$),
and that the scaling group is trivial if and only if
$r_k$ is finite for all $k$

The proofs of the following results are now straightforward. 

\begin{thm}\label{p:5stronglyRegular}For $d =2$ and also for $d=3$
there are $2$ strongly regular meshes, up to conformal affine  isomorphism, with dilation group contained in $\bQ$, namely
\[
\N_{\rm tri}\otimes \bQ, \quad  \N_{\bZ^2}\otimes \bQ
\]
and
\[
 \N_{\rm Bcu}\otimes \bQ,\quad  \N_{\bZ^3}\otimes \bQ
\]
\end{thm}

\begin{thm}\label{p:triscalingcong}Let $\N$ be a regular string-node net and let $\N\otimes \bQ(n)$ and $\N\otimes \bQ(m)$ be the regular meshes associated with the supernatural numbers $n, m$. Then these meshes are conformally isomorphic if and only if $n$ and $m$ are finitely equivalent. 
\end{thm}

Finally, let us define a \emph{pure mesh} in $\bR^d$ to be a strongly regular mesh such that every regular submesh $\M'$ in $\M$ is conformally isomorphic to $\M$. 

\begin{thm}The following statements are equivalent.

(i) $\M$ is a pure mesh in dimension $2$ or $3$.

(ii) $\M$ is conformally affinely isomorphic one of the regular meshes  of Theorem \ref{t:regularmesh} where $F$ (resp $F'$) is equal to $\bQ(n)$, where 
$n = p^\infty$ for some prime $p$ (resp. odd prime $p$).
\end{thm}




\subsection{Mesh extensions}We now give a construction scheme to obtain a minimal enlargement of a given string-node net, $\N$ say, to a string-node net $\N_U$ whose dilation group contains a countable subgroup $U$ of $(\bR_+,\times)$.


We refer to any net
$\N'=(N', S')$ in $\bR^d$ with the properties $N\subseteq N'$ and $|\N| \subseteq |\N'|$ as an \emph{extension} of $\N$. This general construction, with $U$ a dense group in $\bR_+$, provides one way of generating meshes from discrete nets. 

The following lemma plays a role in defining the strings of the net $\N_U$.

\begin{lem}\label{l:coverlemma}
Let $\I$ be a countable family of nondegenerate closed intervals in $\bR$. Then there is a family $\C$ of disjoint closed intervals in $\bR$ with the following properties.

(i) $\C$ is a cover of $\I$ in the sense that every interval of $\I$ lies in an interval of $\C$.

(ii) $\C$ is a minimal disjoint cover in the sense that if $\C'$ is a  
cover of $\I$ by a family of disjoint closed intervals, and  $I \subseteq C' \in \C'$ then $C \subseteq C'$ where $C$ is the interval of $\C$ with $I\subseteq C$.
\end{lem}

\begin{proof}Note that the intervals of a finite or countable sequence of intervals in $\I$  with consecutively pairwise nonvoid intersections, necessarily  lie in the same interval of any disjoint cover. However, the closures of the unions of maximal sets of intervals which are intersection-connected in this manner may fail to be disjoint. The following coarser equivalence relation leads to the existence of $\C$.

Let $\sep(A,B)$ denote the separation distance between two sets in $\bR$ given by  the infimum of distances $|a-b|$ with $a\in A, b\in B$.
Define the equivalence relation $\sim_\I$ on $\bR$ by the requirement that $x \sim_\I y$  if for every $\epsilon >0$ there exists a finite set of intervals $I_1, \dots , I_N$ from $\I$ so that the separation distances
\[
\sep(\{x\},I_1), \sep (I_1, I_2),\dots , \sep(I_{N-1}, I_N), \sep(I_N,\{y\})
\]
are all no greater  than $\epsilon$. We call such a finite set of intervals in $\I$ an $\epsilon$-chain.

Note that if $x_k\sim_\I y$ and $x_k \to x$ as $k\to \infty$ then $x \sim_\I y$ and it follows that the equivalence classes $[x]$ are closed sets. Also each equivalence class $[x]$ is  convex. For if this were not the case then there exist $y, z$ in $[x]$ with $y<z$ and an intermediate point $w$ which is not in $[x]$ and so, by the closedness of $[x]$,  $w$ is contained in a finite closed  interval $J$ disjoint from $[x]$, of width $\delta>0$ say. On the other hand since $x\sim_\I y$ for each $\epsilon_m = 1/m,$ with $m \in \bZ$ and $1/m < \delta$, there exists a set $J_m$ in an $\epsilon$-chain for $x, y$ with $J_m\cap J\neq \emptyset$. It follows from the compactness of $J$ that there is a subsequence $J_{m_k}$ and a point $w\in J$ such that $\sep(w,J_{m_k}) \to 0$ as $k\to \infty$. Thus $w \in [x]$, a contradiction. The lemma follows by taking $\C$ to be the set of equivalence class $[x]$ which are not singleton sets.
\end{proof}

\begin{thm} \label{thm:hull} Let $\N=(N,S)$ be a line segment net in $\bR^d$ and let $U\subseteq \bR_+$ be a countable scaling group for $\mathbb{R}^d$. Then there exists a uniquely determined extension ${\N}_U$ of $\N$ so that
\begin{itemize}
\item[(i)] $U$ is a subgroup of the scaling group $D({\N}_U)$ of ${\N}_U$.
\item[(ii)] If $\N'$ is an extension of the net $\N$ whose scaling group contains $U$ as a subgroup, then $\N'$ is an extension of ${\N}_U$.
\end{itemize}
\end{thm}

\begin{proof}  We first construct the set  $\overline{S}$ of strings of ${\N}_U$. Let $S'=\{d(s_i)|\, s_i\in S, d\in U\}$ be the countable set of all images of strings of $\N$ under the scaling maps in $U$. (There is no need to distinguish repetitions.)
To define the strings in $\overline{S}$ we merge certain colinear elements of $S'$ into single strings in the following manner. For a fixed straight line $g$ let $S'_g$ be the subset of $S'$ of elements that are subsets of  $g$. Define $\overline{S}_g$ to be the set of intervals in the minimal disjoint cover of ${S'}_g$ provided by Lemma \ref{l:coverlemma}. In particular this set of strings has the following properties.

\begin{itemize}
\item[(i)] Each element of $\overline{S}_g$ is a (closed) finite or semi-infinite interval or equal to $g$.
\item[(ii)] Every element of $S'_g$ is contained in one of the sets in $\overline{S}_g$.
\item[(iii)] The intersection of any two distinct subsets in $\overline{S}_g$ is empty.
\end{itemize}

Finally we define $\overline{S}$ to be the union of the sets $\overline{S}_g$ as $g$ runs over the countable set of lines which contain an element of $S'$.

Having constructed the strings we define the set $\ol{N}$ of nodes of ${\N}_U$ to be the set of intersection points of the strings in $\ol{S}$. 

Note that $\ol{N}$ is invariant under the action of $U$ and that $\ol{N}$ contains the orbit $\O_U(N)$ of the set $N$ of nodes of $\N$ under the action of $U$;
\begin{equation}\label{e:nodeinclusion}
\ol{N} \supseteq \O_U(N) = \bigcup_{d\in U}d(N) 
\end{equation}
It is possible that non-parallel strings in $\overline{S}$ intersect at points which do not belong to this orbit in which case the inclusion is proper.

We now show that ${\N}_U=(\overline{S}, \overline{P})$ is a net which satisfies the properties (i) and (ii) in the statement of Theorem~\ref{thm:hull}.

By construction, it is clear that ${\N}_U$ is a string-node  net and that ${\N}_U$ is an extension of $\N$. By construction, it is also clear that ${\N}_U$ satisfies the minimality condition (ii) of Theorem~\ref{thm:hull}. It remains to show that every element of $U$ is an element of the scaling group of ${\N}_U$.

Let $g_1$ and $g_2$ be two distinct straight lines with $g_2=d(g_1)$ for some $d\in U$. Since $U$ is a group it follows that $S'\cap g_1 = d(S'\cap g_2)$. Also $d$ is a bi-Lipschitz map  and so it follows that $d$ effects a bijection between the disjoint closed covers of these sets of intervals. Thus $d$ induces a bijection $\ol{S} \to \ol{S}$ and hence a bijection $\ol{N} \to \ol{N}$, as required.
\end{proof}

In fact the construction above also  applies if $U$ is any countable set of affine transformations.

The construction of an extension of a linear discrete net by a countable scaling group is more straightforward in that minimal disjoint covers are not required. Interesting linear meshes  of this type in two dimensions may be obtained as the projections of regular linear meshes in $\bR^3$  where the projections of a triple of periodicty vectors have incommensurate lengths. In such examples we note that the inclusion of Equation (\ref{e:nodeinclusion}) is proper.

\begin{rem}We comment briefly on some categorical aspects of dense meshes.

Recall first that an infinite bar-joint framework $\G= (G,p)$ can be viewed as a particular  placement or realisation of the infinite underlying graph $G$, sometimes referred to as the \emph{structure graph} of $\G$ or, in reticular chemistry, as the underlying \emph{topology} of $G$.  Similarly, a discrete string-node net $\N=(N,S)$ has an evident underlying structure graph.

On the other hand there is no discrete structure graph which underlies a string-node mesh $\M$. The appropriate associated "pre-metric" structure can be considered to be the topological space $X(\M)=(|\M|,\tau)$  where the topology $\tau$ has a base of open sets provided by the sets $\gamma([a,b])$
in $\bR^d$ arising from parametrisations of the strings and parameter values $a<b$. In fact, although, for simplicity, we have chosen to define strings as nonoverlapping subsets, one could equally well (from the perspective of rigidity and flexibility, for example) allow strings to overlap, and indeed, admit all sets $\gamma([a,b])$ between nodes as strings. 
With this perspective the topological space $X(\M)$  underlying the mesh $\M$ supports a Borel measure, defined by the "placement" $\M$, namely the \emph{string length
measure} $\mu_\M$, determined by the strings in the canonical way. 

We note that the mesh $\M=(N,S)$ also gives rise naturally to a metric space $(|\M|,d)$.
The metric $d$ is given by
\[
d(x,y) = \inf (d_{\gamma_1}(x_1,x_2) +d_{\gamma_2}(x_2,x_3)+\dots +
d_{\gamma_r}(x_{r-1},x_r))
\]
where the infimum is taken over all sets of points 
$x=x_1,x_2,\dots ,x_{r}=y$,
where consecutive points lie on the same strings, $L_1,\dots ,L_{r-1}$ respectively, 
and where $d_{\gamma_i}(\cdot, \cdot )$ is  the string distance separation for
points on the same string $L_i$. 
The metric space $(|\M|, d)$, the topological space $(|\M|, \tau)$
and the Borel measure space $(|\M|, \mu_\M)$ can be regarded as "internal to the mesh" in the sense that their various isomorphism classes do not change under the deformation motions we consider below.
\end{rem}

\section{Flexibility and rigidity}\label{s:flexibility}

We now consider dynamical aspects of meshes, in the sense of their continuous motions and smooth motions which preserve all the internodal distances as measured along each string. Since the nodes of a mesh that lie on a particular string $l$ are dense in $l$ this means that the position (or placement) of a string at a time value $t$ is given by a length preserving map from $l$ to $\bR^d$. 
The terms \emph{flex} and \emph{motion} are used interchangeably. Also there will be no need in the following discussion to consider infinitesimal flexes or general velocity fields.
 
\begin{defn} \label{d:placement}Let $\M=(N,S)$ be a line segment mesh in  $\bR^{d}$.

(i) A \emph{placement} (resp. $C_k$-\emph{smooth placement}) of $\M$  is 
an injective map  
$p : |\M|\to \bR^d$ such that the restriction of $p$ to each string is continuous (resp. $C_k$-smooth) and is length preserving.

(ii) A \emph{continuous flex} of  $\M$ is a pointwise continuous path $t \to p_t, t\in [0,1]$, of placements of $\M$.

(iii) A \emph{$C_k$-smooth flex} of  $\M$ is a pointwise continuous path $t \to p_t, t\in [0,1]$, of
$C_k$-smooth placements of $\M$.

(iv) A \emph{uniformly $C_k$-smooth} flex of $\M$ is a  $C_k$-smooth flex $t \to p_t, t\in [0,1]$,
such that the associated map from $|\M|\times [0,1]$ to $\bR^d$ 
is a $k$-fold continuously differentiable function when restricted to the set $(|\M|^-)^o\times [0,1]$.
\end{defn}

Note that we assume that placements are defined by injective maps
and in particular these are \emph{noncrossing} placements in the sense that the placed strings cannot intersect at nonnodal points.
Also, unlike the usual freedom adopted for bar-joint frameworks, a continuous flex in the sense above admits no "collisions". Thus a continuous flex of a string-node net or mesh may be viewed as a homotopy in the ambient space,  as in the category of knots.

We have defined continuous flexes and smooth flexes of a line segment mesh without direct reference to smooth meshes. One may also regard such a smooth flex as a smooth path in the space of smooth meshes. However we need not consider  this level of abstraction.

\subsection{Planar mesh motions}
Consider first the planar bounded grid mesh $(\N_{\bZ^2}\otimes \bQ)\cap [0,1]^2$, which we now denote more simply as $\M_q$.
We say that a map $f:[0,1]^2\to \bR^2$ is a \emph{laminar map} if it has the form
\[
f(s,t) = \alpha_1(s)+\alpha_2(t)
\]
where the sum is a vector sum for the maps $\alpha_1, \alpha_2: [0,1]\to \bR^2$.
A \emph{laminar placement} of $\M_q$ is a placement $p:|\M_q| \to \bR^2$ which is the restriction of a laminar map.

Note that a laminar map is string-length preserving for $\M_q$
if the maps $\alpha_1, \alpha_2$ are length preserving.
Simple examples show that a  laminar map need not be injective and so need not define a placement. However,  we have the following sufficient condition.

\begin{lem}
Let $\alpha_1, \alpha_2$ be continuously differentiable and length preserving and suppose that the unit vectors $\alpha_1'(s), \alpha_2'(t)$ satisfy the angle constraints
\[
\alpha_1'(s)\cdot (1,0) > 1/\sqrt{2},\quad  \alpha_2'(t)\cdot (0,1) >1/\sqrt{2},\quad s, t \in (0,1).
\]
Then the associate laminar map $f:[0,1]^2\to \bR^2$ determines a $C_1$-smooth placement of $\M_q$.
\end{lem}

\begin{proof}
Let $\alpha_i(t)=(\alpha_{i,x}(t), \alpha_{i,y}(t))$ for $i=1,2$. By the hypotheses we have
\[
\alpha_{1,x}'(s)> \frac{1}{\sqrt{2}} > \alpha_{1,y}'(s)
\] 
and so $\alpha_{1,x}(s)> \alpha_{1,y}(s)$ for $s\in (0,1)$. It follows that for any $0\leq s_1<s_2$ the vector $\alpha_1(s_2)-\alpha_1(s_1)$ lies in the open  cone $y >|x|$.

It follows, similarly, that for $0<t_1<t_2$ the vector $\alpha_2(t_2)-\alpha_2(t_1)$ and its negative lie in the open cone
$|y|<|x|$.

Suppose now that $f(s_1,t_1)=f(s_2, t_2)$ with $s_1\leq s_2$ and $t_1\neq t_2$. Then
$\alpha_1(s_2)-\alpha_1(s_1) = -(\alpha_2(t_2)-\alpha_2(t_1))$ which is a contradiction.
\end{proof}

\begin{defn}
A continuous flex $\gamma=(\gamma_t)$ of the  grid mesh $\M_q$ is a \emph{laminar flex} if each placement map $\gamma_t$ is laminar. 
\end{defn}

From the lemma it follows that there are diverse uniformly smooth flexes of $\M_q$ which are of laminar type and which are determined by the motions of the two boundary strings
that lie in the $x$-axis and the $y$-axis.
The next theorem shows that  uniformly smooth flexes are necessarily of this form for a nonzero time interval.

\begin{lem}\label{l:laminar}
Let $p:[0,1]^2\to \bR^2$ be a continuous map  with the following properties.

(i) $p$ is differentiable on $(0,1)^2$,

(ii) $p$ preserves the lengths of all vertical and horizontal line segments of the form $\{r\}\times [a,b]$ and $ [a,b]\times \{r\}$ with $r\in \bQ$, 

(iii) the vectors
$\frac{\partial p}{\partial x}(x,y)$, $\frac{\partial p}{\partial y}(x,y)$ are
not collinear for all $(x,y)\in [0,1]^2$.


Then $p(x,y)$ has the laminar  form $p(x,y)= \alpha_1(x)+\alpha_2(y)$.
\end{lem}

\begin{proof}
Let $p(s,t)= (f(s,t),g(s,t))$. Note that for fixed $s$ and $t$ the line segment from $(s,t)$ to $(s,t+\delta t)$ has length $|\delta t|$.
It follows from the assumptions of string-length preservation
that the derivative with respect to $t$
of the vector valued function $t \to p(s,t)$ at $(s,t)$ is a vector of magnitude $1$. Thus
\[
(\frac{\partial}{\partial t}f(s,t))^2 +(\frac{\partial}{\partial t}g(s,t))^2 =1
\]
Differentiating with respect to $s$, we obtain
\[
2{(\frac{\partial}{\partial t}p(s,t))\cdot (\frac{\partial}{\partial s}\frac{\partial}{\partial t}p(s,t)) } =0
\]
for all $s, t$ in the open unit square.


By the hypotheses it also follows that there is a similar equation with the roles of $s$ and $t$ exchanged, namely
\[
2(\frac{\partial}{\partial s}p(s,t))\cdot (\frac{\partial}{\partial t}\frac{\partial}{\partial s}p(s,t)) =0.
\]

Since, by (iii), the tangent vectors for the two strings through $(x,y)$
are not colinear we deduce that
\[
\frac{\partial}{\partial x}\frac{\partial}{\partial y}p(x,y) =0
\]
throughout the open unit square.
Since $p$ is differentiable here it follows, on repeated integration, that $p(x,y)$ is laminar on this open set, and  the desired conclusion follows. 
\end{proof}

\begin{thm}\label{t:laminar} Let  $p = (p_t)$ be a
uniformly $C_1$-smooth flex of $\M_q$. Then there is $\delta >0$ such that $p_t$ is laminar for $0\leq t\leq \delta$.
\end{thm}

\begin{proof}
At time $t=0$ the placement $p_t$ is the identity map and so the tangential vectors
$\frac{\partial}{\partial x}p_t(x,y), \frac{\partial}{\partial y}p_t(x,y)$ are orthogonal for all $(x,y)$ in the unit square $[0,1]^2$. Since the flex is uniformly smooth it follows that there exists $\delta >0$ such that for all $0\leq t\leq\delta$ the tangential vectors
$\frac{\partial}{\partial x}p_t(x,y), \frac{\partial}{\partial y}p_t(x,y)$ are not colinear for all $(x,y)$ in $[0,1]^2$.  Thus, by the lemma, all the placements for such $t$ are laminar.
\end{proof}

We consider next the triadic kagome mesh. 
Recall that the infinite periodic kagome bar-joint framework admits diverse continuous motions which are periodic in some sense. 
See \cite{kap-tre-tho-gue}, \cite{owe-pow-nyjm-1} for example. We now show that in contrast to the rectangular case above, this flexibility does not carry over in any way to the triadic kagome mesh.


\begin{thm}\label{t:triadicrigid}
The triadic kagome mesh is rigid with respect to uniformly $C_1$-smooth flexes.
\end{thm}

\begin{proof}
We first show that any $C_1$-smooth placement of 
$\M_{\rm kag}$ is congruent to the identity placement of $\M_{\rm kag}$. 
Note that such a placement, $p_t$ say, has a
unique continuous extension of $p_t: \bR^2\to \bR^2$. Moreover, a restriction of this extension defines a string-length preserving placement of the containing triadic triangular mesh, $\M_{\rm tri}$. It suffices then to show that the $C_1$-smooth  mesh $\M = p_t(\M_{\rm tri})$, defined by the range of this  placement, is congruent to $\M_{\rm tri}$.

Let  $0<t<1$ and let $q_0$ be a node of the mesh $\M$.  By the  $C_1$-smoothness of the strings of $\M$ there are $3$ three unit tangent vectors to the three  strings through $q_0$, say $a_1, a_2, a_3$. We may choose these vectors so that the vectors $a_1, a_2, a_3, -a_1, -a_2, -a_3$ correspond to a consecutive sequence of points on the unit circle, with the angles between $a_1, a_2$ and between $a_2, a_3$ acute angles. 
We claim that all the consecutive angles are equal to $\pi/3$. To see this it is sufficient to note that no consecutive angle is greater than $\pi/3$. We now  show that this follows from the string-length preservation property applied to cross-strings connecting nodes on adjacent strings through $q_0$. 

Let $q_0'$ be the node in $\M_{\rm tri}$ which is the preimage of $q_0$ under $p_t$, and let $q_1'$ and $q_2'$ be nodes on adjacent strings through $q_0'$ which lie on a common cross string.
The internodal string lengths agree with the Euclidean distances, $\|q_1' - q_2'\|, \|q_1'-q_0'\|$ and $\|q_2'-q_0'\|$ respectively, and these lengths are equal. 
Let $q_1, q_2$ be the images of $q_1', q_2'$ under $p_t$. Since $p_t$ preserves string lengths we have $d_{01}=d_{12}=d_{02}$ where $d_{ij}$ is the string distance between $q_i$ and $q_j$ for the (unique) string $l_{ij}$ of $\M$ through $q_i, q_j$.

Suppose now, by way of contradiction, that angle between the tangent vectors at $q_0$ for the strings $l_{01}$ and $l_{02}$ is greater than $\pi/3$. Then 
\medskip

(i) the ratio $\|q_1-q_2\|/\|q_1-q_0\|$ has a finite  limit  greater than $1$ as $q_1$ tends to $q_0$.
\medskip

Also, by string-length preservation,
\medskip

(ii) the ratio $d_{01}/\|q_1-q_0\|$ has limit $1$, as $q_1$ tends to $q_0$,
\medskip

and, by the triangle inequality,
\medskip

(iii)   $d_{12}$  is not less than $\|q_1-q_2\|$.
\medskip

\medskip

Combining the above it follows that if  $q_1$ is sufficiently close to $q_0$ then
\[
\frac{d_{12}}{d_{01}} = 
\frac{d_{12}}{\|q_1-q_2\|}
\frac{\|q_1-q_2\|}{\|q_1-q_0\|}
\frac{\|q_1-q_0\|}{d_{01}} >1
\]
which is a contradiction.

To see that $\M_{\rm tri}$ is rigid for uniformly smooth flexes we may note that $\M_{\rm tri}$ contains the $4$-valent mesh $\M$ whose strings are either parallel to the $x$-axis
or at angle $\pi/3$ to this axis. This is simply a skew form of the grid mesh and the argument of Lemma \ref{l:laminar} applies to show that a uniformly smooth flex of $\M$ is laminar for small $t$. However, it is straightforward to show that a $C_1$-smooth laminar placement of $\M$ which is angle preserving at the nodes is necessarily a rigid placement and so the completion of the proof follows.

\end{proof}

We next show that there are simple inductive constructions that lead to dense line segment meshes $\M$ in the plane which are highly flexible.

Let us say that a line segment mesh $\M$ is \emph{locally flexible} if there exists a bounded open set $U$ in $\bR^2$ and a path of homeomorphisms $p_t$ of $\bR^2$, for $t\in [0,1]$, such that
$p=(p_t)$ is a nontrivial continuous (string-length preserving) flex of $\M$ with $p_t(x,y)=(x,y)$ for all points $(x,y)$ in the complement of $U$.

\medskip

In fact there is sufficient freedom in the construction of such meshes below that we can arrange that each of the placements $p_t$, when restricted to any line segment string, is the restriction of an isometry. Another way of expressing this is to introduce the class of \emph{rod-pin meshes}. Let us define a rod-pin mesh somewhat informally as
a line segment string-node mesh for which the strings are inflexible rods. As with finite rod-pin frameworks, rods may be joined at their endpoints or at interior points and so  there are $3$ forms of joining nodes for a pair of rods: a double end point node, a single end point node, or a  double interior point node. A flex of a rod-pin mesh is then a continuous flex of the associated line segment string-node mesh 
such that the restriction to each string  is isometric.

\begin{thm}\label{t:locallyflexiblemesh} Let $r\geq 3$ be a positive integer. Then there exists a locally flexible dense rod-pin mesh in $\bR^2$ such that the degree of every node is $r$.
\end{thm} 

\begin{proof}
Consider a cyclic rod-pin framework $\G_0= (G_0, p)$ with $4$ bars and $4$ joints as in the left hand side diagram of Figure \ref{f:stretchy}.
We construct a bounded rod-pin mesh $\M$, such that (i) $\partial |\M|=|\G_0|$, (ii) the closure of $|\M|$ is $|\G_0|$ together with the interior region, and
(iii) $\M$ has a nontrivial  boundary fixed continuous flex.

\begin{center}
\begin{figure}[ht]
\centering
\includegraphics[width=8cm]{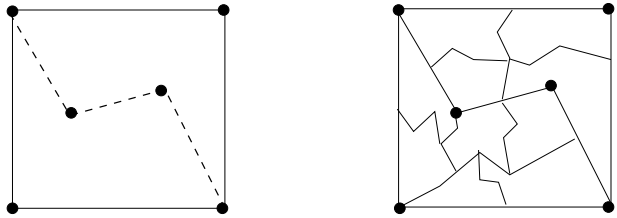}
\caption{Constructing a locally flexible rod-pin mesh.}
\label{f:stretchy}
\end{figure}
\end{center}

Step $n= 1$: Add a rod-pin path framework (dotted in the figure) which lies in the bounded region interior to the boundary and has the following properties. The added path of rods has its end joints  pinned at $2$ joints on  the boundary, there are at least $2$ other joints, and there is a nonconstant continuous flex $p(t)$ of the combined rod-pin framework which fixes the $4$ boundary bars. A boundary joint for this path may be an existing joint for $\G_0$ (as illustrated) or may be a new joint which is interior to one of the rods on the boundary. Let $p(t), t\in [0,1]$, be any nontrivial flex of the resulting framework $\G_1$ which fixes the outer boundary and for which there are no self intersections of nodes and joints.

Steps $n = 2, 3, \dots $ : Repeat Step 1 for the two open regions, to obtain frameworks $\G_3 \supset \G_2 \supset \G_1$, after which
(Steps $n = 4, 5, \dots $) repeat such divisions for the resulting regions, and so on, ad infinitum, according to the following requirements.

(a) The flex $p(t), t\in [0,1]$, of $\G_{n-1}$  extends to a flex of $\G_n$, also denoted $p(t)$, of the augmented framework. Realising this property is elementary since the added rod-pin path can possess an arbitrary finite number of nodes and we are only required to obtain a motion of the augmented path which can remain interior to the appropriate region of the complement of $(G_{n-1}, p(t))$ at time $t$.

(b) The maximum of the diameters of the regions at Step $n$, and the maximum of the areas of the regions at Step $n$, tend to zero as $n$ tends to infinity. This can be assured by adding a rod-pin path to reduce a maximum diameter, or reduce a maximum area, when $n$ is divisible by $3$.

(c) The following \emph{joint exhaustion process} is adopted. The joints introduced at Step $n$ are labelled as additions to a sequence $v_1, \dots ,v_{m_n}$. 
If the degree of $v_i$ is equal to $r$ then this joint is not revisited, meaning that this joint is not used as an end point for a subsequent rod-pin path. If, on the other hand, $v_j$ is the first joint with degree less than $r$, and if $n$ is congruent to $1$ mod $3$,  then this joint is revisited.

(d) If $n$ is congruent to $2$ mod $3$ then one of the end joints of the rod-pin path is chosen to be the midpoint of a pre-existing rod with maximum length.

This process of construction defines a rod-pin mesh together with a continuous flex $p(t)$ which fixes the boundary and the theorem follows from this.
\end{proof}

\begin{rem}
We remark that one can be more liberal in the definition of a general string-node net and admit continuous strings
which are nonrectifiable, as in the case of fractal curves. In the extreme case such meshes may be \emph{uniformly nonrectifiable} in the sense that for each of the strings $\gamma: [0,1]\to \bR^2$ the curves $\gamma([a,b])$ are nonrectifiable for all $a<b$.
One can readily construct such meshes by a process of \emph{string interpolation}, analogous to that in the proof above.
Alternatively we note that interesting divisible uniformly nonrectifiable meshes arise from tilings by \emph{nonrectifiable rep-tiles}. By the latter we mean a compact set  $K$ with interior which may be tiled by congruent copies of $rK$ for some $r<1$, and  where the boundary is a non-rectifiable closed curve. The twin dragon is a well known example. See for example Vince \cite{vin}.
Note  that there are no longer any string length constraints, so for vacuous reasons such meshes, including the twin dragon mesh, are locally flexible.
\end{rem}

\section{Further directions}
We indicate a number of areas for further development.

\subsection{Block meshes} Figure \ref{f:block} is indicative of the rational grid mesh $\M$ in the unit cube, that is the three dimensional version of the square grid of Theorem \ref{t:laminar}. It is evident that a literal variant of this theorem does not hold  for this block mesh. Indeed,
the corresponding laminarity condition, defined in terms of a three-fold vector sum, might be referred to as \emph{strong laminarity}. To see how this may fail for all small time intervals note first that a bounded 2D mesh in a plane in $\bR^3$ admits smooth ruled surface motions analogous to that of a sheet of paper. Evidently one can repeat such motions in a parallel manner to obtain smooth (string-length-preserving) flexes of $\M$ that are not strongly laminar. It seems plausible that for small enough time intervals a smooth flex will agree with such a motion.

\begin{center}
\begin{figure}[ht]
\centering
\includegraphics[width=4cm]{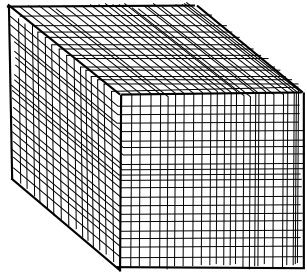}
\caption{Some strings of the block mesh.} 
\label{f:block}
\end{figure}
\end{center}

The noncrossing requirement in our definition of continuous and smooth flexes, which ensures proper nonintersecting motions, presents some interesting general configuration space questions in the case of flexible bounded meshes. Recalling that a bounded mesh $ \M$ is necessarily connected we may define the \emph{minimum diameter} $d_{\rm min}(\M)$ to be the infimum of the diameters of the placements $p(|\M|)$ of $\M$ occurring in any continuous flex. In a similar way one can define the \emph{maximum diameter} $d_{\rm max}(\M)$, which is generally  easier to compute. What are the minimum diameters for the block mesh and the grid mesh $\M_q$? Evidently  the block mesh has minimum diameter of at most $\sqrt{2}$ in view of elementary vertical concertina motions which collapse the height while keeping the base fixed.  

\subsection{The Sierpinski mesh} The Sierpinski triangle mesh $\M_{\rm Sier}$ is defined to be the line segment mesh whose strings are the line segments that appear in the construction of the Sierpinski triangle. 
See Figure \ref{f:sier}. 
With the exception of the three extremal nodes the nodes have degree $4$ and have the same local geometry, up to rotation. Indeed each such node has a cycle of  $4$ adjacent  angles between the incident strings, namely
$\pi/3, \pi /3, \pi/3$ and $\pi$, with only the first and third of these being infinitesimally triangulated, in the sense given in the proof of Theorem \ref{t:triadicrigid}.
It follows that for any placement one obtains only a partial subconformality condition, namely that the two corresponding angles are no greater than $\pi/3$. Thus further argument is required to show that the $C_1$-smooth placements of $\M_{\rm Sier}$ are trivial, if indeed this is the case. We leave as open problems the determination of whether the Sierpinski mesh is smoothly flexible or continuously flexible.

\begin{center}
\begin{figure}[ht]
\centering
\includegraphics[width=4cm]{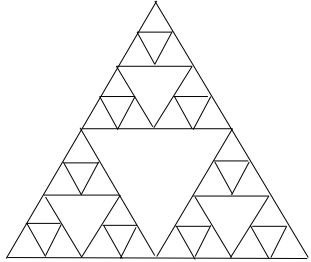}
\caption{Some strings of the Sierpinski mesh $\M_{\rm Sier}$.}
\label{f:sier}
\end{figure}
\end{center}

One can similarly construct fractal meshes based on iterated function systems where once again string-length preserving proper motions seem to be trivial.
We also remark that one can construct dense line-segment meshes with self-similarity properties by "filing in" tilings of the plane of quasicrystallographic type defined by  partition and expansion rules associated with a finite set of tiles.
In such a case it is usually possible to define an associated \emph{substitution rule mesh} by also carrying out the construction in an inward manner. In these inward substitutions one may use the same subdivision rule. 
In this way one may naturally define  the \emph{pinwheel mesh} for example.

\subsection{Semidiscrete structures} In an interesting 2007 thesis, E. B. Ashton \cite{ash} introduced the subject of \emph{continuous tensegrities} and analysed various examples and their infinitesimal and continuous rigidity. 
Recall, informally, that a finite tensegrity  can be viewed as a modified bar-joint framework $(G,p)$ in which some of the bars are replaced by strings (cables) and some of the bars are replaced by struts (incompressible bars that admit stretching). In the case of continuous tensegrities the finiteness of the graph is relaxed and continua of strings and bars are possible, as in the ensuing example.
The methodologies used for these constraint systems have a number of commonalities with our discussions. However,  the consistent focus in \cite{ash} is to obtain generalisations of the Roth-Whitely theorem \cite{rot-whi} to the effect that a sufficient condition for infinitesimal rigidity is the existence of a positive self-stress together with the rigidity of the underlying framework $(G,p)$.

The following illustrative example appears in Chapter 4 of \cite{ash}. Let the points of the unit circle about the origin in the plane provides a set of joints. Let the antipodal joints be connected by struts and define (line segment) strings to connect two joints that are separated by a subarc of length $h/\pi$ units. If the subarc
distance is irrational, then the continuous tensegrity is a single
connected tensegrity which is infinitesimally rigid.

Some methodological commonalities can also be found in the detailed analysis of the flexibility of semidiscrete frameworks and joined ribbons given by Karpenkov \cite{kar}. In this setting a single ribbon is in effect a string-node net with a continuum of rigid strings (rods) connecting corresponding points on two continuous strings (being the boundary of the ribbon). Karpenkov has 
obtained the striking result that generic double  ribbons have one-dimensional configuration spaces and curious continuous motions. 

The rigidity and flexibility analysis we have given above can be viewed as the beginnings of what we see as a novel and interesting topic of geometric rigidity. As well as the fundamental problem of determining the forms of rigidity and flexibility of various meshes in dimensions $2$ and $3$, there are also natural further directions inspired in part by the discrete setting, such as the analysis of
infinitesimal flexibility (Badri et al \cite{bad-et-al}),
the implications of symmetry (Ross et al \cite{ros-et-al}), and the implications of boundary conditions (Theran et al \cite{the-et-al}).


\begin{thebibliography}{99}

\bibitem{ash} E.B. Ashton, Exploring continuous tensegrities, PhD thesis, University of Georgia, 2007

\bibitem{asi-rot} L. Asimow and B. Roth,
{The rigidity of
graphs}, Trans. Amer. Math. Soc., 245 (1978), 279-289.


\bibitem{asi-rot-2} L. Asimow and B. Roth, The rigidity of graphs II, J. Math. Anal. Appl., 68 (1979) 171-190.

\bibitem{bad-et-al} G. Badri, D. Kitson and S.C. Power,
The  almost  periodic  rigidity  of  crystallographic bar-joint frameworks. Symmetry, 6
(2014), 308-328.




\bibitem{cox-PLMS} H. S. M. Coxeter, Regular skew polyhedra in three and four dimensions, and their topological
analogues, Proc. London Math. Soc., 43 (1937), 33-62.

\bibitem{cox-laves}H. S. M. Coxeter, On Laves' graph of girth ten, Canadian Journal of Mathematics 7 (1955), 18-23.



\bibitem{del-et-al-1}
O. Delgado-Friedrichs, M. O'Keeffe and O. M. Yaghi,
$3$-periodic nets and tilings: regular and semiregular nets,
Acta Cryst A, 59 (2003), 22-27.


\bibitem{del-et-al-2}
O. Delgado-Friedrichs, M. O'Keeffe and O. M. Yaghi,
$3$-periodic nets and tilings: semiregular  nets,
Acta Cryst A, 59 (2003), 513-525.

\bibitem{del-et-al-2.5}
O. Delgado-Friedrichs, M. O'Keeffe and O. M. Yaghi,
Three-periodic nets and tilings: edge-transitive
binodal structures, Acta Cryst. A62 (2006), 350-355.

\bibitem{del-oke} 
O. Delgado-Friedrichs, M. O'Keeffe. Crystal nets as graphs: terminology and definitions. J. Solid State Chem. 178 (2005), 2480-2485. 


\bibitem{del-et-al-3} O. Delgado-Friedrichs, M. D. Foster, M. O’Keeffe,  D. M. Proserpio, M. M.J. Treacy, O. M. Yaghi, What do we know about three-periodic nets? J. of Solid State Chemistry,  178 (2005), 2533-2554.





\bibitem{gru-lostlectures} B. Gruenbaum and G.C. Shephard, Lectures on Lost Mathematics, Notes for the Special Session on Rigidity Theory at the
AMS meeting in Syracuse, NY, 1978;https:digital.lib.washington.edu

\bibitem{hyd-et-al} S.T. Hyde, M. O'Keeffe and D. M. Proserpio, A short history of an elusive yet ubiquitous structure in chemistry, materials, and mathematics, Angewandte Chemie International Edition 47 (42) (2008) 7996-8000, 





\bibitem{kap-tre-tho-gue} V. Kapko, M. M. J. Treacy, M. F. Thorpe and  S. D. Guest,
On the collapse of locally isostatic networks,
Proc. of the Royal Society A, 465 (2009), 3517-3530.

\bibitem{kar} O. Karpenkov,  Finite and infinitesimal flexibility of semidiscrete surfaces.  Arnold Mathematical Journal, 1 (2015), 403-444.


\bibitem{lav} F. Laves, Zur Klassifikation der Silikate, Z. Kristallogr., 82 (1932), 1-14.



\bibitem{rcsr} M. O'Keeffe, A.M. Peskov, S.J. Ramsden and O.M. Yaghi,  
The Reticular Chemistry Structure Resource (RCSR) database of, and symbols for, crystal nets,
Acc. Chem. Res. 2008, 41, 1782-1789



\bibitem{owe-pow-nyjm-1}
 J.C. Owen and S.C. Power, Infinite bar-joint frameworks, crystals
and operator theory, New York J. Math.,  17 (2011), 445-490.



\bibitem{pel-sch}    D. Pellicer and E. Schulte,
        Polygonal Complexes and Graphs for Crystallographic Groups,
        Rigidity and Symmetry, Volume 70 of the series Fields Institute Communications, 2014, pp 325-344.


\bibitem{pow-poly} S.C. Power, Polynomials for crystal frameworks and the rigid unit mode spectrum, Phil. Trans. R.
Soc. A 2014, 372, doi:10.1098/rsta.2012.0030.


\bibitem{rad} C. Radin, The pinwheel tiling of the plane, Ann. of Math., 139 (1994), 661-701. 





\bibitem{ros-et-al} E. Ross, B. Schulze and  W. Whiteley, Finite motions from periodic frameworks with added symmetry, International Journal of Solids and Structures 48 (2011), 1711-1729.


\bibitem{rot-whi} B. Roth and W. Whiteley, Tensegrity frameworks,
Trans. Amer. Math. Soc., 265 (1981), 419-446.

\bibitem{sch-et-al}
G.E. Schr\"{o}der-Turk, S. Wickham, H. Averdunk, F. Brink, J.D. Fitzgerald, L. Poladian, M.C.J. Large and S.T. Hyde,
The chiral structure of porous chitin within the wing-scales of Callophrys rubi, Journal of Structural Biology, 174 (2011), 290-295.

\bibitem{schulte} E. Schulte, Polyhedra, complexes, nets and symmetry,
Acta Cryst. A70 (2014), 203–216

\bibitem{sne} K. Snelson, Forces Made Visible, Hudson Hills Press Inc., U.S.A., 2009.

\bibitem{sun-notices} T. Sunada, Crystals that nature might miss creating, Notices of the Amer. Math. Soc., 55 (2008), 208-215


\bibitem{sun-book}T. Sunada, Topological crystallography,
with a view towards discrete geometric analysis, Surveys and Tutorials in the Applied Mathematical Sciences, Volume 6, Springer, 2013.


\bibitem{the-et-al}
L. Theran, A. Nixon, E. Ross, M. Sadjadi, B. Servatius and  M. Thorpe, 
Anchored boundary conditions for locally isostatic networks, 
Physical Review E 92 (5), 053306 (2015).

\bibitem{vin}A. Vince, Digit tiling of Euclidean space, in Directions in Mathematical Quasicrystals, CRM Monographs, eds M. Baake and R.V. Moody, vol 13, 2000.


\bibitem{whi-handbook} W. Whiteley, Rigidity and Scene Analysis; in the Handbook of Discrete and Computational Geometry, J. Goodman and J. O'Rourke (ed.), CRC Press, 1997, 893-916.
\end{thebibliography}
\end{document}